\documentclass[a4paper,pdftex]{amsart}
\usepackage{amsmath,amssymb,amsbsy,amsthm,amsfonts,comment,graphicx,color,tikz,bm,mathdots,enumitem}
\usepackage[all]{xy}
\usetikzlibrary{calc,hobby}
\usepackage{hyperref}
\hypersetup{
    colorlinks=true,
    linkcolor=purple,
    urlcolor=orange,
    citecolor=teal,
    pdfauthor = {Ryuma Orita and Kanon Yashiro},
    pdftitle = {Morse--Bott--Smale chain complex},
    pdfpagemode = UseNone,
    bookmarksnumbered=true,
}

\title{Morse--Bott--Smale chain complex}

\author{Ryuma Orita}
\address[Ryuma ORITA]{Department of Mathematics, Faculty of Science, Niigata University, Niigata 950-2181, Japan}
\email{\href{mailto:orita@math.sc.niigata-u.ac.jp}{orita@math.sc.niigata-u.ac.jp}}
\urladdr{\url{https://ryuma-orita.netlify.app/}}

\author{Kanon Yashiro}
\address[Kanon YASHIRO]{Graduate School of Science and Technology, Niigata University, Niigata 950-2181, Japan}
\email{\href{mailto:yashiro@m.sc.niigata-u.ac.jp}{yashiro@m.sc.niigata-u.ac.jp}}
\urladdr{\url{https://sites.google.com/view/kanon-yashiro/home}}

\thanks{This work was supported by JSPS KAKENHI Grant Number 21K13787.}

\newtheorem{theorem}{Theorem}[section]
\newtheorem{lemma}[theorem]{Lemma}
\newtheorem{proposition}[theorem]{Proposition}
\newtheorem{corollary}[theorem]{Corollary}

\theoremstyle{definition}
\newtheorem{definition}[theorem]{Definition}
\newtheorem{example}[theorem]{Example}

\theoremstyle{remark}
\newtheorem{remark}[theorem]{Remark}

\newcommand{\ZZ}{\mathbb{Z}}

\newcommand{\RR}{\mathbb{R}}

\newcommand{\CB}{\mathrm{CB}}
\newcommand{\HB}{\mathrm{HB}}
\newcommand{\CM}{\mathrm{CM}}
\newcommand{\HM}{\mathrm{HM}}
\newcommand{\Crit}{\mathrm{Crit}}

\newcommand{\ind}{\operatorname{ind}}
\newcommand{\grad}{\operatorname{grad}}
\newcommand{\Image}{\operatorname{Im}}
\newcommand{\Ker}{\operatorname{Ker}}

\newcommand{\relmiddle}[1]{\mathrel{}\middle#1\mathrel{}}

\makeatletter
\newcommand{\dotDelta}{{\vphantom{\Delta}\mathpalette\d@tD@lta\relax}}
\newcommand{\d@tD@lta}[2]{\ooalign{\hidewidth$\m@th#1\mkern-1mu\cdot$\hidewidth\cr$\m@th#1\Delta$\cr}}
\makeatother

\begin{document}

\begin{abstract}
Banyaga and Hurtubise \cite{BH10} defined the Morse--Bott--Smale chain complex
as a quotient of a large chain complex by introducing five degeneracy relations.
However, their five degeneracy relations are in fact redundant \cite{BH25}.
In the present paper, we unify these five conditions into a single degeneracy condition and resolve the issue of the well-definedness of the Morse--Bott--Smale chain complex.
This provides an appropriate definition of the Morse--Bott homology and more computable examples.
Moreover, we show that our chain complex for a Morse--Smale function is quasi-isomorphic to the usual Morse--Smale--Witten chain complex.
As a consequence, we obtain an alternative proof of the Morse Homology Theorem.
\end{abstract}

\maketitle

\tableofcontents


\section{Introduction}\label{section:Introduction}

Morse theory is a powerful tool to reveal the topology of a smooth manifold
by analyzing the isolated critical points of a smooth function, called a Morse function.
A Morse--Bott function generalizes a Morse function
by allowing its critical point set to be a finite disjoint union of connected submanifolds \cite{Bo54}.
Such functions are useful, for instance, in the presence of a group action.
The (co)homology theory for Morse--Bott functions admits several constructions,
such as Austin and Braam’s model \cite{AB95},
Fukaya's model \cite{Fu96},
and the cascades model \cite{Bourgeois2002,Fr03}.
More recently, Zhou \cite{Zhou2024} unified the three methods
and constructed the \textit{minimal} Morse--Bott cochain complex for every Morse--Bott flow category.
We refer the reader to his paper \cite{Zhou2024} for a excellent comparison of these constructions and more details.

Banyaga and Hurtubise \cite{BH10} also defined a chain complex,
which they called the \textit{Morse--Bott--Smale chain complex},
for a Morse--Bott function satisfying a certain transversality condition (Definition \ref{definition:MBS trans.}).
Their chain complex is constructed as a quotient of the smooth singular cubical chain complex of the critical submanifolds,
modulo five degeneracy relations introduced in \cite[Definition 5.9]{BH10}.
An advantage of their construction is that the Morse--Bott--Smale chain complex for a Morse--Smale function
coincides with the Morse--Smale--Witten chain complex (see Section \ref{Section:Morse homology} for the definition).
In the case of a constant function, it coincides with the smooth singular cubical chain complex.
However, since the five degeneracy relations overlap,
it is difficult to compute concrete examples of Morse--Bott homology
and to establish the well-definedness of the boundary operator.
As clarified in \cite{BH25}, their five degeneracy relations are in fact redundant.
For details, see Remark \ref{remark:private_communication}.

In the present paper,
we resolve the issue of well-definedness and unify Banyaga and Hurtubise's degeneracy relations into a single condition (Definition \ref{definition:deg con})
by introducing the notion of a \textit{good representing chain system} (Definition \ref{definition:good rep.sys.}).
Roughly speaking, a good representing chain system is a collection of representative chains chosen for all the spaces under consideration.
This provides an appropriate definition of the Morse--Bott--Smale chain complex and more computable examples,
see Section \ref{section:computation}.
The crucial point is that,
due to the change in the degeneracy conditions,
our chain complex,
which is also called the \textit{Morse--Bott--Smale chain complex},
for a Morse--Smale function \textit{no longer} coincides with the Morse–Smale–Witten chain complex.
However, one can show that these two chain complexes are quasi-isomorphic (Theorem \ref{theorem:Morse_MB_isom}).
Moreover, our Morse--Bott--Smale chain complex for a constant function coincides
with the smooth singular \textit{simplicial} chain complex up to sign (Example \ref{example:Constant func.}).

Furthermore, one can show that
our Morse--Bott homology does not depend on the choices of the Morse--Bott--Smale function
nor the good representing chain system (Theorem \ref{theorem:MB homology theorem}).
This provides another proof of the Morse Homology Theorem (Theorem \ref{theorem:Morse homology Theorem}).


\section*{Acknowledgements}

The authors would like to express their sincere gratitude to Augustin Banyaga and David Hurtubise for carefully reading the first version of the present paper and for engaging in thoughtful and generous discussions with us.


\section{Chain complexes of topological chains}

In this section, we set conventions and notation to define our Morse--Bott--Smale chain complex.
We introduce chain complexes of \textit{abstract/singular topological chains}.
We closely follow the arguments in \cite[Sections 2--4]{BH10}.


\subsection{Abstract topological chains}\label{section:abstract topological chains}

Let $p$ be an integer.
For $p\geq 0$ we fix a family $C_p$ of topological spaces.
Let $S_p$ be the free abelian group generated by the elements of $C_p$, i.e., $S_p=\ZZ[C_p]$.
For $p<0$ or $C_p=\emptyset$ we set $S_p=0$. 

\begin{definition}[{\cite[Definition 4.1]{BH10}}]\label{definition:abs.top.cpx.}
For $p\geq 1$ let $\partial_p\colon S_p \to S_{p-1}$ be a homomorphism satisfying the following conditions:
\begin{enumerate}
    \item For any $P\in C_p$ we have $\partial_p(P) =\sum_k n_k P_k$
    where $n_k=\pm 1$ and $P_k\in C_{p-1}$ is a subspace of $P$ for all $k$.
    \item $\partial_{p-1}\circ\partial_p=0$.
\end{enumerate}
For $p\leq 0$ we set $\partial_p=0$.
Then the pair $(S_*,\partial_*)$ is called a \textit{chain complex of abstract topological chains}.
\end{definition}

\begin{example}\label{example:face abs. simpl. cpx.}
Fix a large positive integer $N$.
Let $\Delta^N\subset\RR^{N+1}$ denote the standard $N$-simplex.
Namely,
\[
    \Delta^N = \left\{\, (t_0,t_1,\ldots,t_N)\in\RR^{N+1} \relmiddle| t_0+t_1+\dots+t_N=1\ \text{and}\ t_i\geq 0\ \text{for all}\ i \,\right\}.
\]
For $p=0$,~1, \dots,~$N$ let $\dotDelta^p$ denote a $p$-face of the standard simplex $\Delta^N$.
Let $\iota\colon\Delta^p\to\dotDelta^p\subset\Delta^N$ be the standard embedding.
We define $C_p^{\Delta}$ to be the set of the $p$-faces of $\Delta^N$
and let $S_p^{\Delta}=\ZZ[C_p^{\Delta}]$.
For $p<0$ or $p>N$ we set $S_p^{\Delta}=0$.
We define a boundary operator
$\partial_p\colon S_p^{\Delta}\to S_{p-1}^{\Delta}$
by the formula
\[
    \partial_p(\dotDelta^p)=\sum_{i=0}^p (-1)^i \iota\left(\left.\iota^{-1}(\dotDelta^p)\right|_{t_i=0}\right),
\]
where $\dotDelta^p\in C^{\Delta}_p$ and $(t_0,\ldots,t_p)\in\Delta^p$. 
Then we deduce that $\partial_{p-1}\circ\partial_p=0$ holds.
Now the pair $(S_*^{\Delta},\partial_*)$ is a chain complex of abstract topological chains.
\end{example}


\subsection{Singular topological chains}\label{section:singular topological chains}

Let $(S_*,\partial_*)$ be a chain complex of abstract topological chains induced from a collection $\{C_p\}_{p\geq 0}$.
Let $X$ be a topological space.

\begin{definition}[{\cite[Definition 4.2]{BH10}}]
Let $p\geq 0$.
A \textit{singular $C_p$-space} in $X$ is a continuous map $\sigma\colon P\to X$ where $P\in C_p$.
Let $S_p(X)$ be the free abelian group generated by the singular $C_p$-spaces in $X$.
For $S_p=0$ or $X=\emptyset$ we set $S_p(X)=0$.
\end{definition}

For $p\geq 1$ the boundary operator $\partial_p\colon S_p\to S_{p-1}$ induces a boundary operator
$\partial_p\colon S_p(X)\to S_{p-1}(X)$ as follows.
For a singular $C_p$-space $\sigma\colon P\to X$
we define
\[
    \partial_p(\sigma)=\sum_k n_k \sigma|_{P_k},
\]
where $\partial_p(P)=\sum_k n_k P_k$.
For $p\leq 0$ we set $\partial_p=0$.
Then the pair $(S_*(X),\partial_*)$ is called a \textit{chain complex of singular topological chains in $X$}.


\subsection{The fibered product of singular topological chains}

Let $P_1$, $P_2$, and $X$ be topological spaces.
Let $\sigma_1\colon P_1\to X$ and $\sigma_2\colon P_2\to X$ be continuous maps.
We recall that the fibered product of $\sigma_1$ and $\sigma_2$ is defined to be
\[
    P_1\times_X P_2 = \left\{\, (x_1, x_2)\in P_1\times P_2 \mid \sigma_1(x_1)=\sigma_2(x_2) \,\right\} .
\]
If $P_1$, $P_2$, and $X$ are smooth manifolds of dimension $p_1$, $p_2$, and $b$, respectively,
and if $\sigma_1$ and $\sigma_2$ are smooth and $\sigma_1$ is transverse to $\sigma_2$,
then the fibered product $P_1\times_X P_2$ is a smooth manifold of dimension $p_1+p_2-b$
(see, e.g., \cite[Lemma 4.5]{BH10}).

For a collection $\{C_p\}_{p\geq 0}$ of families of topological spaces,
an element $P\in C_p$ is said to have \textit{degree} $p$.
Furthermore, if $X$ is a manifold of dimension $b$,
then we associate the degree $p_1+p_2-b$ to the fibered product $P_1\times_X P_2$
where $P_1\in C_{p_1}$ and $P_2\in C_{p_2}$.

\begin{definition}[{\cite[Definition 4.6]{BH10}}]\label{definition:boundary operator to fiber product}
Let $(S_*,\partial_*)$ be the chain complex of abstract topological chains induced from a collection $\{C_p\}_{p\geq 0}$.
Assume that the collection $\{C_p\}_{p\geq 0}$ is closed under taking the fibered product with respect to some collection of maps.
For $i=1$,~2 if $\sigma_i=\sum_k n_{i,k} \sigma_{i,k}\in S_{p_i}(X)$,
where $\sigma_{i,k}\colon P_{i,k}\to X$ is a singular $C_{p_i}$-space for all $k$,
then the \textit{fibered product} of $\sigma_1$ and $\sigma_2$ over $X$ is defined to be
\[
    P_1\times_X P_2 = \sum_{k,\,j} n_{1,k} n_{2,j} P_{1,k}\times_X P_{2,j} ,
\]
where $P_1=\sum_k n_{1,k} P_{1,k}\in S_{p_1}$ and $P_2=\sum_j n_{2,j} P_{2,j}\in S_{p_2}$.
The boundary operator for the fibered product is defined to be
\[
    \partial_{p_1+p_2-b}(P_1\times_X P_2) = \partial_{p_1}(P_1)\times_X P_2 + (-1)^{p_1+b} P_1\times_B\partial_{p_2}(P_2).
\]
If either $\sigma_1=0$ or $\sigma_2=0$,
then we define $P_1\times_X P_2=0$.
\end{definition}

According to \cite[Lemma 4.7]{BH10},
$(P_1\times_{X_1} P_2) \times_{X_2} P_3$ and $P_1\times_{X_1} (P_2\times_{X_2} P_3)$
are the same as abstract topological chains. 
Hence the fibered product of singular topological chains is compatible with its structure as an abstract topological chain complex.
Namely, under the setting of Definition \ref{definition:boundary operator to fiber product},
$(S_*(X),\partial_*)$ becomes a chain complex of singular topological chains.


\section{Morse functions and Morse--Bott functions}

We quickly review the construction of the Morse--Smale--Witten chain complex
and the definition of Morse--Bott functions.
We refer the reader to an excellent book \cite{BH04a} for more details.
Let $(M,g)$ be a closed oriented Riemannian manifold.


\subsection{The Morse--Smale--Witten chain complex}\label{Section:Morse homology}

Let $f\colon M\to\RR$ be a Morse--Smale function on $(M,g)$.
Namely, we assume that $f$ is a Morse function,
and the stable manifolds and the unstable manifolds
(with respect to the negative gradient vector field $-\grad_g{f}$) of all critical points of $f$ intersect transversally.
Let $\Crit(f)$ denote the critical point set of $f$.
Moreover, for each integer $k\geq 0$
let $\Crit_k(f)$ be the set of critical points of Morse index $k$.

For $k=0$,~1, \dots,~$\dim{M}$ we define $\CM_k(f)=\ZZ[\Crit_k(f)]$.
For $k<0$ or $k>\dim{M}$ we set $\CM_k(f)=0$.
We define the boundary operator $\partial^{\mathrm{Morse}}_k\colon\CM_k(f)\to\CM_{k-1}(f)$ by the formula
\[
    \partial^{\mathrm{Morse}}_k(q) = \sum_{p\in\Crit_{k-1}(f)} n(q,p)\, p,
\]
where $q\in\Crit_k(f)$ and $n(q,p)$ is the finite number of negative gradient flow lines from $q$ to $p$ with signs determined by orientations of the unstable manifolds and the given orientation of $M$.

\begin{theorem}[{The Morse Homology Theorem (cf.\ \cite[Theorem 7.4]{BH04a}}]\label{theorem:Morse homology Theorem}
The pair $(\CM_*(f),\partial^{\mathrm{Morse}}_*)$ is a chain complex.
Moreover, for each $k\in\ZZ$ the $k$-th homology $\HM_k(f)=H_k\bigl((\CM_*(f),\partial^{\mathrm{Morse}}_*)\bigr)$
is isomorphic to the singular homology $H_k(M;\ZZ)$ of $M$.
\end{theorem}

Theorem \ref{theorem:Morse homology Theorem} implies that
the Morse homology $\HM_k(f)$ does not depend on the choices of the Morse--Smale function $f$,
the Riemannian metric $g$, nor the orientations of the unstable manifolds.
In Section \ref{section:A proof of the MH thm},
we provide another proof of Theorem \ref{theorem:Morse homology Theorem}.


\subsection{Morse--Bott--Smale functions}

We recall the definition of Morse--Bott--Smale functions,
which are Morse--Bott functions satisfying a certain transversality.

Let $f\colon M\to\RR$ be a smooth function on $M$.
We assume that the critical point set $\Crit(f)$ decomposes into
a finite disjoint union of connected submanifolds $B_1$, \dots,~$B_n$ in $M$,
called \textit{critical submanifolds}.
For any $\alpha=1$, \dots,~$n$ and any $p\in B_{\alpha}$
we have the decomposition
\[
    T_p M \cong T_p B_{\alpha}\oplus\nu_p(B_{\alpha}),
\]
where $\nu_p(B_{\alpha})$ denotes the normal space of $B_{\alpha}$ in $M$ at $p$.
Then the Hessian $\mathrm{Hess}_p(f)\colon T_p M \times T_p M \to \RR$ induces
a symmetric bilinear form on the normal space $\nu_p(B_{\alpha})$,
which is denoted by $\mathrm{Hess}^\nu_p(f)$.

\begin{definition}\label{definition:MB}
A smooth function $f\colon M\to \RR$ is said to be \textit{Morse--Bott}
if the critical point set of $f$ is a finite disjoint union of connected submanifolds of $M$
and for any critical submanifold $B$ the bilinear form $\mathrm{Hess}^\nu_p(f)$ is non-degenerate for all $p \in B$.
\end{definition}

We refer the reader to \cite[Theorem 3.51]{BH04a} and \cite{BH04b} for a proof of the following lemma.

\begin{lemma}[The Morse--Bott Lemma]
Let $f\colon M\to \RR$ be a Morse--Bott function and let $B$ be a critical submanifold.
For any $p \in B$ there exists a coordinate neighborhood $U$ of $p$ and a decomposition of the normal bundle of $B$ restricted to $U$:
$\nu_*(B)|_U = \nu_*^+(B) \oplus \nu_*^-(B)$.
Moreover, we have
\[
    f(x)=f(u,v,w)=f(B)+\|v\|^2-\|w\|^2,
\]
for any $x=(u,v,w)\in U$ where $u\in B$, $v\in\nu_*^+(B)$ and $w\in\nu_*^-(B)$.
\end{lemma}

We assume that $f$ is Morse--Bott.
Let $B$ be a critical submanifold and $p\in B$.
By the Morse--Bott lemma, the tangent space $T_pM $ decomposes into the direct sum
\begin{equation}\label{eq:decomp}
    T_p M \cong T_p B \oplus \nu_p^+(B) \oplus \nu_p^-(B).
\end{equation}
Then we define the \textit{Morse--Bott index} $\ind{B}$ of the critical submanifold $B$ to be the dimension of the subspace $\nu_p^-(B)$.

We fix a Riemannian metric $g$ on $M$.
Let $\{\varphi_t\}_{t\in\RR}$ denote the flow of $-\grad_g{f}$.
For $p\in\Crit(f)$ let $W^s(p)$ (resp.\ $W^u(p)$) denote the stable (resp.\ unstable) manifold of $p$.
We also define the stable manifold and the unstable manifold of $B$ by
$W^s(B) = \bigcup_{p\in B} W^s(p)$ and $W^u(B) = \bigcup_{p\in B} W^u(p)$, respectively.

\begin{theorem}[{\cite[Theorem A.9]{AB95}, cf.\ \cite[Theorem 3.3]{BH10}}]\label{theorem:SU mfd for MB}
The stable and unstable manifolds $W^s(B)$ and $W^u(B)$ are
the images of injective immersions $E^+\colon \nu_*^+(B) \to M$ and $E^-\colon \nu_*^-(B) \to M$, respectively.
There are smooth endpoint map $\mathrm{ev}_+\colon W^s(B) \to B$ and beginning map $\mathrm{ev}_-\colon W^u(B) \to B$
given by $\mathrm{ev}_+(x)=\lim_{t\to\infty} \varphi_t(x)$ and $\mathrm{ev}_-(x)=\lim_{t\to -\infty} \varphi_t(x)$,
which have the structure of fiber bundles.
\end{theorem}

Now we define Morse--Bott--Smale functions.

\begin{definition}\label{definition:MBS trans.}
A Morse--Bott function $f\colon M \to \RR $ is called \textit{Morse--Bott--Smale}
if for any two critical submanifolds $B$, $B'$
and any point $p\in B$ the submanifolds $W^u(p)$ and $W^s(B')$ intersect transversally.
\end{definition}

Let $B$ and $B'$ be critical submanifolds of $f$ of dimensions $b$ and $b'$, respectively.
Then we have $\dim W^u(B) = b + \ind{B}$ and $\dim W^s(B') = \dim{M} - \ind{B'}$.
We set
\[
    W(B,B') = W^u(B) \cap W^s(B').
\]
If $f$ is Morse--Bott--Smale,
then the intersection $W(B,B')$ is a submanifold of $M$ of dimension $\ind{B} - \ind{B'} + b$,
provided that $W(B,B')\neq\emptyset$.
We note that $\dim{W(B,B')}$ is independent of $\dim{B'}$.
Moreover, if $B\neq B'$ and $W(B,B')\neq\emptyset$, then we have $\ind{B}>\ind{B'}$, see \cite[Lemma 3.6]{BH10}.
Therefore, Morse--Bott--Smale functions are \textit{weakly self-indexing} in the sense of \cite{AB95}.
Namely, Morse--Bott index is strictly decreasing along negative gradient flow lines.


\subsection{Compactified moduli spaces}


Let $f\colon M\to\RR$ be a Morse--Bott--Smale function on $(M,g)$.
Let $\{\varphi _t\}_{t\in\RR}$ denote the flow of the negative gradient vector field $-\grad_g{f}$.


\subsubsection{Gluing and compactification}

For any two critical submanifolds $B$ and $B'$ of $f$
the flow $\{\varphi_t\}_{t\in\RR}$ defines a free $\RR$-action on $W(B,B')$.
Let
\[
    \mathcal{M}(B,B') = W(B,B')/\RR
\]
denote the quotient space,
called the moduli space of (negative) gradient flow lines from $B$ to $B'$.
The following theorems are crucial, see \cite[\S A.3]{AB95} and \cite[Theorems 4.8 and 4.9]{BH10}) for details.

\begin{theorem}[Gluing]\label{theorem:Gluing}
Let $B$, $B'$ and $B''$ be critical submanifolds of $f$.
Then there exist $\varepsilon >0$ and an injective local diffeomorphism
\[
    G\colon \mathcal{M}(B,B') \times_{B'} \mathcal{M}(B',B'') \times (0,\varepsilon) \to \mathcal{M}(B,B'')
\]
onto an end of $\mathcal{M}(B,B'')$.
\end{theorem}

\begin{theorem}[Compactification]\label{theorem:Compactification}
Let $B$ and $B'$ be distinct critical submanifolds of $f$.
Then the moduli space $\mathcal{M}(B,B')$ admits a compactification $\overline{\mathcal{M}}(B,B')$,
which consists of all the piecewise gradient flow lines from $B$ to $B'$.
Moreover, $\overline{\mathcal{M}}(B,B')$ is either empty or a compact smooth manifold with corners
of dimension $\ind{B} -\ind{B'}+b-1$.
Furthermore, the beginning and endpoint maps $($Theorem \ref{theorem:SU mfd for MB}$)$ extend to smooth maps
$\mathrm{ev}_-\colon\overline{\mathcal{M}}(B,B') \to B$ and $\mathrm{ev}_+\colon \overline{\mathcal{M}}(B,B') \to B'$,
where the beginning point map $\mathrm{ev}_-$ has the structure of a fiber bundle.
\end{theorem}

For any two critical submanifolds $B$ and $B'$,
define a partial ordering $B\succ B'$  if $B\neq B'$ and there exists a negative gradient flow line from $B$ to $B'$.
Note that if $B\succ B'$ and $B'\succ B''$, then $B\succ B''$ by Theorem \ref{theorem:Gluing}.
Moreover, a compactified moduli space $\overline{\mathcal{M}}(B,B'')$ can be described as 
\[
    \overline{\mathcal{M}}(B,B'') = \mathcal{M}(B,B'') \cup \bigcup_{B\succ B'\succ B'' } \overline{\mathcal{M}}(B,B') \times_{B'} \overline{\mathcal{M}}(B',B'').
\]
Therefore, $\overline{\mathcal{M}}(B,B'')$ consists of all \textit{broken} gradient flow lines from $B$ to $B''$.
For a topology on $\overline{\mathcal{M}}(B,B'')$ which is compatible with the topology of fibered products induced by the product topology, we refer the reader to \cite[Definition 4.10]{BH10}.


\subsubsection{Compactified moduli spaces as abstract topological chains}

We define the structure of a chain complex of abstract topological chains for the collection of fibered products of compactified moduli spaces.

Let $i=0$,~1, \dots,~$\dim{M}$.
Let $B_i$ be the set of critical points of $f$ of Morse--Bott index $i$.
Namely, the set $B_i$ is the disjoint union of critical submanifolds of Morse--Bott index $i$.
To simplify the notation,
we assume that for each $i=0$,~1, \dots,~$\dim{M}$ the components of $B_i$ have the same dimension.
Let $b_i=\dim{B_i}$.

\begin{definition}[{\cite[Definition 4.11]{BH10}}]\label{definition:deg of moduli}
Let $j=0$,~1, \dots,~$i$ and consider the compactified moduli space $\overline{\mathcal{M}}(B_i,B_{i-j})$.
We define the degree of $\overline{\mathcal{M}}(B_i,B_{i-j})$ to be 
$\deg{\overline{\mathcal{M}}(B_i,B_{i-j})}=j+b_i-1$.
Moreover, we define the boundary of $\overline{\mathcal{M}}(B_i,B_{i-j})$ by the (component-wise) formal sum
\[
    \partial_{j+b_i-1}\bigl(\overline{\mathcal{M}}(B_i,B_{i-j})\bigr)%
    = (-1)^{i+b_i} \sum_{i-j<n<i} \overline{\mathcal{M}}(B_i,B_n) \times_{B_n} \overline{\mathcal{M}}(B_n,B_{i-j}),
\]
where the fibered product is taken over the beginning and endpoint maps $\mathrm{ev}_-$ and $\mathrm{ev}_+$
(see Theorem \ref{theorem:Compactification}).
If $B_n=\emptyset$,
then we set $\overline{\mathcal{M}}(B_i,B_n)=\overline{\mathcal{M}}(B_n,B_{i-j})=0$.
Furthermore, the boundary operator extends to fibered products of compactified moduli spaces
via Definition \ref{definition:boundary operator to fiber product}.
\end{definition}

Let $p$ be an integer.
For $p\geq 0$ let $C_p^{\mathcal{M}}$ be the set of the $p$-dimensional connected components of compactified moduli spaces
and the $p$-dimensional connected components of fibered products of the form
\[
    \overline{\mathcal{M}}(B_{i_1},B_{i_2}) \times_{B_{i_2}} \overline{\mathcal{M}}(B_{i_2},B_{i_3}) \times_{B_{i_3}} \cdots \times_{B_{i_{n-1}}} \overline{\mathcal{M}}(B_{i_{n-1}},B_{i_n}),
\]
where $\dim{M} \geq i_1 > i_2 > \cdots > i_n \geq 0$
and the fibered product is taken with respect to the beginning and endpoint maps $\mathrm{ev}_-$ and $\mathrm{ev}_+$.
Let $S_p^{\mathcal{M}} = \ZZ[C_p^{\mathcal{M}}]$.
For $p<0$ or $C_p^{\mathcal{M}} = \emptyset$ we set $S_p^{\mathcal{M}} = 0$. 

\begin{lemma}[{\cite[Lemma 4.12]{BH10}}]\label{lemma:boundary operator on compactified moduli space}
The pair $(S_*^{\mathcal{M}},\partial_*)$ is a chain complex of abstract topological chains.
\end{lemma}

We need the following result for our purpose, see Remark \ref{remark:fibered_product_ori_corners}.

\begin{lemma}[{\cite[Lemma 5.21]{BH10}}]\label{lemma:fiber product is closed on Man(smooth, cpt, with corners)}
Let $B$ and $B'$ be critical submanifolds.
Let $P$ be a compact smooth manifold with corners.
If $\sigma \colon P \to B$ is a smooth map, then the fibered product $P \times _{B} \overline{\mathcal{M}}(B,B')$ is a compact smooth manifold with corners.
\end{lemma}


\subsubsection{Orientations}

In the rest of the section,
we discuss orientations on compactified moduli spaces
(see also \cite[Section 5.2]{BH10}).
We assume that all the critical submanifolds of $f$
and their negative normal bundles are oriented.
Let $B$ be a critical submanifold.
First we orient the positive normal bundle $\nu^+(B)$ by the decomposition \eqref{eq:decomp}
and orient the stable and unstable manifolds $W^s(B)$, $W^u(B)$
(see Theorem \ref{theorem:SU mfd for MB}).
Moreover, for any two critical submanifolds $B$ and $B'$
we orient $W(B,B')=W^u(B)\cap W^s(B')$ by the decomposition
\begin{align*}
    T_x M%
    &\cong T_x W(B,B') \oplus \nu_x\bigl(W(B,B')\bigr) \\
    &\cong T_x W(B,B') \oplus \nu_x\bigl(W^s(B')\bigr) \oplus \nu_x \bigl(W^u(B)\bigr).
\end{align*}
Then we orient the moduli space $\mathcal{M}(B,B')$
by the decomposition
\[
    T_x W(B,B') = \mathrm{span}\bigl(-(\grad_g{f})_x\bigr) \oplus T_x \mathcal{M}(B,B').
\]
By Theorem \ref{theorem:Gluing},
the boundary of the compactified moduli space $\overline{\mathcal{M}}(B,B')$ consists of fibered products of compactified moduli spaces.
We orient these fibered products as follows.

\begin{definition}[{\cite[Definition 5.7]{BH10}}]\label{definition:orientation of fiber product}
Let $X$ be an oriented manifold.
Let $P_1$ and $P_2$ be oriented manifolds with corners.
Let $\sigma_1\colon P_1\to X$ and $\sigma_2\colon P_2\to X$ be smooth maps.
Assume that $\sigma_1$ and $\sigma_2$ intersect transversally and stratum transversally
(see \cite[Section 5.5]{BH10} or \cite{Ni82} for the definition).
Then we define an orientation on the manifold $P_1\times_X P_2$ with corners
by the formula
\[
    (-1)^{\dim{X}\dim{P_2}} T_x(P_1 \times_X P_2) \oplus (\sigma_1\times\sigma_2)^*\bigl(\nu_*(\Delta_X)\bigr) = T_*(P_1\times  P_2),
\]
where $\Delta_X$ is the diagonal of the product $X\times X$.
\end{definition}

This definition states that the gluing maps are orientation reversing,
which can be deduced by comparing the orientations provided above using \cite[Proposition 2.7]{Br93}.

\begin{lemma}[{\cite[Lemma 5.8]{BH10}}]\label{lemma:orientations are associative}
Under the setting of Definition \ref{definition:orientation of fiber product},
the orientations on fibered products of transverse intersections of smooth manifolds with corners is associative, i.e.,
\[
    (P_1 \times_{X_1} P_2) \times_{X_2} P_3 = P_1 \times_{X_1} (P_2 \times_{X_2} P_3)
\]
as oriented manifolds with corners.
\end{lemma}


\section{Morse--Bott homology}\label{section:MB homology}

Let $(M,g)$ be a closed oriented Riemannian manifold.
Let $f\colon M\to\RR$ be a Morse--Bott--Smale function on $(M,g)$.
We assume that all the critical submanifolds of $f$ and their negative normal bundles are oriented.
In this section, we define the Morse--Bott--Smale chain complex $(\CB_*(f),\bm{\partial}_*)$ for $f$,
which is a quotient complex of a larger chain complex $(\widetilde{\CB}_*(f),\bm{\partial}_*)$.

For $i=0$,~1, \dots,~$\dim{M}$ let $B_i\subset \Crit(f)$ be the set of critical points of index $i$ and let $N>\dim{M}$.
To simplify the notation,
we assume that the components of $B_i$ have the same dimension.


\subsection{Larger chain complex}\label{subsection:simplicial MBS chain complex}

For any $p\geq 0$
let $C^f_p$ be the set consisting of $p$-faces of the $N$-simplex $\Delta^N$
and $p$-dimensional connected components of fibered products of the form
\[
    \dotDelta^q%
    \times_{B_{i_1}} \overline{\mathcal{M}}(B_{i_1},B_{i_2})%
    \times_{B_{i_2}} \overline{\mathcal{M}}(B_{i_2},B_{i_3})%
    \times_{B_{i_3}} \dots%
    \times_{B_{i_{n-1}}} \overline{\mathcal{M}}(B_{i_{n-1}},B_{i_n}),
\]
where $\dotDelta^q$ is a $q$-face of $\Delta^N$ for some $q\leq p$, and $\dim{M} \geq i_1 > i_2 > \cdots > i_n \geq 0$.
Here the fibered products are taken with respect to a smooth map $\sigma\colon \dotDelta^q\to B_{i_1}$
and the beginning and endpoint maps $\mathrm{ev}_-$ and $\mathrm{ev}_+$.
Let $S^f_p=\ZZ[C^f_p]$.
For $p<0$ or $C^f_p=\emptyset$ we set $S^f_p=0$.
Let $\partial\colon S^f_p\to S^f_{p-1}$ denote the boundary operator defined by
Example \ref{example:face abs. simpl. cpx.},
Definition \ref{definition:boundary operator to fiber product}, and
Definition \ref{definition:deg of moduli}.
For instance,
\[
    \partial\left(\dotDelta^q \times_B \overline{\mathcal{M}}(B,B')\right)%
    =\partial(\dotDelta^q) \times_B \overline{\mathcal{M}}(B,B') + (-1)^q \dotDelta^q \times_B \partial\bigl(\overline{\mathcal{M}}(B,B')\bigr),
\]
where $B$ and $B'$ are critical submanifolds.

\begin{remark}
In \cite[Section 5.1]{BH10},
Banyaga and Hurtubise considered faces of the cube $[0,1]^N$ instead of faces of the standard simplex $\Delta^N$.
\end{remark}

\begin{remark}\label{remark:fibered_product_ori_corners}
Since any face of $\Delta^N$ is a compact smooth manifolds with corners,
each element of $C^f_p$ is also a compact smooth manifold with corners
by Lemma \ref{lemma:fiber product is closed on Man(smooth, cpt, with corners)}.
Furthermore, due to Lemma \ref{lemma:orientations are associative},
one can define a well-defined orientation on each element of $C^f_p$.
\end{remark}

For $i=0$,~1, \dots,~$\dim{M}$
we consider the set $B_i\subset\Crit(f)$ of critical points of index $i$.
As in Section \ref{section:singular topological chains},
the chain complex $(S^f_*,\partial_*)$ of abstract topological chains defines
a chain complex $(S_*(B_i),\partial_*)$ of singular topological chains in $B_i$.
For each $p\geq 0$ we recall that
the group $S_p(B_i)$ is freely generated by the singular $C^f_p$-spaces in $B_i$.

We define the subgroup $\mathcal{S}_p(B_i)$ of $S_p (B_i)$ generated by the smooth maps $\sigma \colon P \to B_i$ satisfying $\sigma=\mathrm{ev}_+\circ\mathrm{pr}$,
provided that $P$ is a connected component of a fibered product, where $\mathrm{pr}$ denotes the projection onto the last component of the fibered product.
For $p<0$ or $B_i=\emptyset$ we set $\mathcal{S}_p(B_i) = S_p(B_i) = 0$.

\begin{definition}\label{definition:Morse--Bott chain group}
An element in $\mathcal{S}_p(B_i)$ is said to have \textit{Morse--Bott degree} $p+i$.
\end{definition}

For each $k\in\ZZ$ we define
\[
    \widetilde{\CB}_k(f) = \bigoplus_{i=0}^{\dim{M}} \mathcal{S}_{k-i}(B_i).
\]

For each $j=0$,~1, \dots,~$\dim{M}$
we define a homomorphism $\partial_{[j]}\colon \mathcal{S}_p(B_i)\to \mathcal{S}_{p+j-1}(B_{i-j})$ as follows.
For $j=i+1$, \dots,~$\dim{M}$ we set $\partial_{[j]}=0$.
For $j=0$ we define $\partial_{[0]}=(-1)^{p+i}\partial_p$
where $\partial_p\colon \mathcal{S}_p(B_i)\to \mathcal{S}_{p-1}(B_i)$ is the boundary operator of the chain complex $(S_*(B_i),\partial_*)$.

Let $j=1$,~2, \dots,~$i$.
Let $\sigma\colon P\to B_i$ be a singular $C^f_p$-space in $\mathcal{S}_p(B_i)$.
We then define $\partial_{[j]}(\sigma)$ to be the composition map
\[
    \mathrm{ev}_+\circ\mathrm{pr}_2\colon P\times_{B_i} \overline{\mathcal{M}}(B_{i},B_{i-j}) \xrightarrow{\mathrm{pr}_2} \overline{\mathcal{M}}(B_{i},B_{i-j}) \xrightarrow{\mathrm{ev}_+} B_{i-j},
\]
where $\mathrm{pr}_2$ is the projection onto the second component
and $\mathrm{ev}_+$ is the endpoint map.
The following lemma says that $\partial_{[j]}(\sigma)\in \mathcal{S}_{p+j-1}(B_{i-j})$.

\begin{lemma}[cf.\ {\cite[Lemma 5.3]{BH10}}]
Let $\sigma\colon P\to B_i$ be a singular $C^f_p$-space in $\mathcal{S}_p(B_i)$.
Then for any $j=1$,~$2$, \dots,~$i$
adding the connected components of $P \times _{B_i} \overline{\mathcal{M}}(B_{i},B_{i-j})$
$($with sign when  the dimension of a component is zero$)$,
we can consider $P\times_{B_i}\overline{\mathcal{M}}(B_{i},B_{i-j})$ as an element of $S^f_{p+j-1}$.
Therefore, for all $j=1,2,\dots,i$
we have an induced homomorphism $\partial_{[j]}\colon \mathcal{S}_p(B_i)\to \mathcal{S}_{p+j-1}(B_{i-j})$.
\end{lemma}

Let $k=1$,~2, \dots,~$\dim{M}$.
We define a homomorphism $\bm{\partial}_k \colon \widetilde{\CB}_k(f) \to \widetilde{\CB}_{k-1}(f)$ by
\begin{equation}\label{eq:boundary_op}
    \bm{\partial}_k = \bigoplus_{p+i=k} \bigoplus_{j=0}^i \left(\partial_{[j]}\colon \mathcal{S}_p(B_i)\to \mathcal{S}_{p+j-1}(B_{i-j})\right).
\end{equation}
Then we can draw the following diagram.
\[
    \xymatrix@C=35pt@R=20pt{
    \ddots & \vdots \ar@{}[d]|{\oplus} &\vdots \ar@{}[d]|{\oplus} &\vdots \ar@{}[d]|{\oplus} & \\
    \cdots & \mathcal{S} _0 (B_2) \ar[r]^-{\partial_{[0]}} \ar[dr]|(.45){\partial_{[1]}} \ar[ddr]|(.3){\partial_{[2]}}|\hole \ar@{}[d]|{\oplus} & 0 \ar@{}[d]|{\oplus} & 0 \ar@{}[d]|{\oplus} & \\
    \cdots & \mathcal{S} _1 (B_1) \ar[r]^(.55){\partial_{[0]}} \ar[dr]|(.45){\partial_{[1]}} \ar@{}[d]|{\oplus}& \mathcal{S} _0 (B_1) \ar[r]^-{\partial_{[0]}} \ar[dr]|(.45){\partial_{[1]}} \ar@{}[d]|{\oplus} & 0 \ar@{}[d]|{\oplus} &  & \\ 
    \cdots & \mathcal{S} _2 (B_0) \ar[r]^{\partial_{[0]}} \ar@{}[d]|{||} & \mathcal{S} _1 (B_0) \ar[r]^{\partial_{[0]}} \ar@{}[d]|{||} & \mathcal{S} _0 (B_0) \ar[r]^-{\partial_{[0]}} \ar@{}[d]|{||} & 0 \\
    \cdots & \widetilde{\CB}_2(f) \ar[r]^{\bm{\partial}_2} & \widetilde{\CB}_1(f) \ar[r]^{\bm{\partial}_1} & \widetilde{\CB}_0(f) \ar[r]^-{\bm{\partial}_0} & 0
    }
\]
One has $\bm{\partial}_{k-1}\circ\bm{\partial}_k=0$ in the following sense.

\begin{proposition}[cf.\ {\cite[Proposition 5.5]{BH10}}]\label{proposition:boundary operator}
For every $j=0$,~$1$, \dots,~$\dim{M}$ we have
\[
    \sum_{q=0}^j \partial_{[q]} \circ \partial_{[j-q]}=0.
\]		
\end{proposition}

\begin{corollary}
The pair $(\widetilde{\CB}_* (f),\bm{\partial}_*)$ is a chain complex.
\end{corollary}


\subsection{Representing chain systems}\label{subsection:rep. sys.}

We define representing chain systems necessary for the construction of our Morse--Bott--Smale chain complex $(\CB_*(f),\bm{\partial}_*)$.
Let $\{C^f_p\}_{p\geq 0}$ be the collection defined in Section \ref{subsection:simplicial MBS chain complex}.

For each $p\geq 0$ we consider the set of pairs
\[
    A_p = \left\{\,(P,s_P) \relmiddle| P\in S^f_{p},\ s_{P}\in S_{p}^{\Delta}(P)\,\right\},
\]
where $S_{p}^{\Delta}(P)$ denotes the $p$-th singular simplicial chain group of $P$.
For topological spaces $P$ and $Q$ we define $S_{p}^{\Delta}(P+Q)=S_{p}^{\Delta}(P)\oplus S_{p}^{\Delta}(Q)$.
We consider $S_{p}^{\Delta}(P)$ and $S_{p}^{\Delta}(Q)$ as subgroups of $S_{p}^{\Delta}(P+Q)$.
Moreover, we set $S_{p}^{\Delta}(-P)=S_{p}^{\Delta}(P)$
and define $S_{p}^{\Delta}(0)=\bigoplus _{P\in C^f_p} S_{p}^{\Delta}(P)$.
We form the free abelian group generated by the elements of $A_p$, and then take the quotient by the relation
\[
    (P,s_P) + (Q,s_Q) = (P+ Q,s_P+s_Q),
\]
for any $(P,s_P)$, $(Q,s_Q)\in A_p$.
We perform this construction for all $p$, then take their direct sum to obtain a group denoted by $\mathcal{Q}$.
Namely,
\[
    \mathcal{Q} = \bigoplus_{p\geq 0} \frac{\ZZ[A_p]}{\langle\, (P,s_P)+(Q,s_Q)-(P+Q,s_P+s_Q) \mid (P,s_P),\, (Q,s_Q)\in A_p \,\rangle}.
\]

\begin{definition}[cf.\ {\cite[Definition 6.7]{BH10}}]\label{definition:rep.sys.}
A \textit{representing chain system} for $f$ is a subgroup $\mathcal{R}$ of $\mathcal{Q}$ that satisfies the following conditions.
\begin{enumerate}
    \item For any $P\in C^f_p$
    there exists an element $(P,s_P)\in\mathcal{R}$
    such that the singular simplicial chain $s_P$ represents the relative fundamental class of $P$ in $H_p(P,\partial P)$.
    \item If $P\in C^f_p$ satisfies $\partial _p (P) = \sum _k n_k P_k \in S^f_{p-1}$
    where $n_k=\pm 1$ and $P_k\in C^f_{p-1}$,
    then
    \[
        \partial _p (s_{P}) = \sum_k n_k s_{P_k} \in S_{p-1}^{\Delta}(P).
    \]
\end{enumerate}
For $(P,s_P)\in\mathcal{R}$ we call $s_P$ a \textit{representing chain} for $P$
and write as $s_P\in\mathcal{R}$ for simplicity.
\end{definition}

\begin{remark}
The second condition in Definition \ref{definition:rep.sys.} ensures that
for every $(P,s_{P})\in\mathcal{R}$ the pair $(\partial_{p}(P),\partial_{p}(s_{P}))$ can be decomposed into elements of that satisfy the first condition.
\end{remark}

The sum of two representing chain systems is again a representing chain system.
Hence all representing chain systems together form a directed partially ordered set.
We focus on minimal representing chain systems satisfying additional conditions.

\begin{definition}\label{definition:good rep.sys.}
A representing chain system $\mathcal{R}$ for $f$ is called \textit{good}
if it satisfies the following conditions.

\begin{enumerate}
    \item For any $P\in C^f_p$ there exists a unique element $(P,s_P)\in \mathcal{R}$
    such that $s_{P}$ represents the positive relative fundamental class of $P$ in $H_p(P,\partial P)$.
    \item Let $P$ be a connected component of a fibered product of the form 
    \[
        \dotDelta \times_{B_{i_1}} \overline{\mathcal{M}}(B_{i_1},B_{i_2}) \times_{B_{i_2}} \cdots \times_{B_{i_{n-1}}} \overline{\mathcal{M}}(B_{i_{n-1}},B_{i_{n}}),
    \]
    where $\dotDelta$ is a face of $\Delta^N$.
    Let $s_{P}=\sum_{\alpha} n_{\alpha} \sigma_{\alpha}\in\mathcal{R}$ be the representing chain for $P$.
    For each $\alpha$ let $\Delta^p_{\alpha}=\Delta^p$ denote the domain of the singular simplex $\sigma_{\alpha}$.
    For any $\alpha$ and any critical submanifold $B\prec B_{i_n}$
    consider the fibered product $\Delta^{p}_{\alpha} \times_{B_{i_{n}}} \overline{\mathcal{M}}(B_{i_{n}},B)$
    of the composition map
    \[
        \mathrm{ev}_+\circ\mathrm{pr}\circ\sigma_{\alpha} \colon \Delta^p_{\alpha} \xrightarrow{\sigma_{\alpha}} P \xrightarrow{\mathrm{pr}} \overline{\mathcal{M}}(B_{i_{n-1}},B_{i_{n}}) \xrightarrow{\mathrm{ev}_+} B_{i_n},
    \]
    and the beginning map $\mathrm{ev}_- \colon \overline{\mathcal{M}}(B_{i_{n}},B) \to B_{i_n}$.
    
    If the representing chain for $\Delta^{p}_{\alpha} \times_{B_{i_{n}}} \overline{\mathcal{M}}(B_{i_{n}},B)$ is of the form
    \[
        s_{\Delta^{p}_{\alpha} \times_{B_{i_{n}}} \overline{\mathcal{M}}(B_{i_{n}},B)} = \sum_{\beta} m_{\beta} \tau_{\beta}, 
    \]
    then the representing chain for the fibered product $P \times_{B_{i_{n}}} \overline{\mathcal{M}}(B_{i_{n}},B)$ coincides with
    \[
        s_{P\times_{B_{i_{n}}} \overline{\mathcal{M}}(B_{i_{n}},B)} = \sum_{\alpha,\,\beta} n_{\alpha} m_{\beta} \bigl((\sigma_{\alpha}\times \mathrm{id})\circ\tau_{\beta}\bigr), 
    \]
    where $\mathrm{id}$ denotes the identity map on $\overline{\mathcal{M}}(B_{i_n},B)$.
\end{enumerate}
\end{definition}

\begin{remark}    
A good representing chain system is required so that the computations in each row of our Morse--Bott--Smale chain complex
coincide with those of the singular simplicial chain complex up to sign,
and also behave in a manner consistent with the homomorphisms $\partial_{[j]}$,
$j=1$,~2, \dots,~$\dim{M}$.
\end{remark}

Hereafter, we choose the standard embedding $\iota\colon\Delta^p\to\dotDelta^p\subset\Delta^N$
as a representing chain for a $p$-face $\dotDelta^p$ of $\Delta^N$.
Namely, $s_{\dotDelta^p}=\iota$.

\begin{lemma}\label{lemma:existence of a good rep. chains}
There exists a good representing chain systems for $f$.
\end{lemma}

\begin{proof}
We construct a good representing chain systems for $f$
by choosing exactly one representing chain for each $P\in C^f_p$.
For a $p$-face $\dotDelta^p$ of $\Delta^N$,
choose the standard embedding $\iota\colon\Delta^p\to\dotDelta^p\subset\Delta^N$ as a representing chain.

Assume that $P$ is a connected component of a fibered product the form $\dotDelta^q \times_{B} \overline{\mathcal{M}}(B,B')$,
where $\dotDelta^q$ is a $q$-face of $\Delta^N$,
$\overline{\mathcal{M}}(B,B')$ is a manifold without boundary, and
\[
    \deg \bigl(\dotDelta^q \times_{B} \overline{\mathcal{M}}(B,B')\bigr)+\ind{B'} \leq N.
\]
In this case, note that $\overline{\mathcal{M}}(B,B')=\mathcal{M}(B,B')$.
Then the boundary of $\dotDelta^q \times_{B} \mathcal{M}(B,B')$ is the sum of
connected components of fibered products of the form $\dotDelta^{p-1} \times_{B} \mathcal{M}(B,B')$.
As an initial step, for such $B$ and $B'$
we choose a representing chain for each connected component of the form $\dotDelta^0 \times_{B} \mathcal{M}(B,B')$.
Each chosen representing chain determines the fundamental class of the corresponding space.
We then extend these representing chains to the representing chains of the relative fundamental classes of the connected components of the form $\dotDelta^1 \times_{B} \mathcal{M}(B,B')$.
By appropriately assigning signs to these extensions,
we then define them as representing chains for those connected components.
By iterating this procedure,
we can define a representing chain for each connected component of the form
$\dotDelta^q \times_{B} \overline{\mathcal{M}}(B,B') = \dotDelta^q \times_{B} \mathcal{M}(B,B')$.

Next we describe how to choose a representing chain for a connected component of a fibered product of the form 
\[
    \dotDelta^q \times_{B_{i_1}} \overline{\mathcal{M}}(B_{i_1},B_{i_2}) \times_{B_{i_2}} \cdots \times_{B_{i_{n-1}}} \overline{\mathcal{M}}(B_{i_{n-1}},B_{i_{n}}).
\]
Put $d=\max\{\, i_j-i_{j+1} \mid j=1,\ldots,n-1\,\}$.

We consider the case of $d=1$.
Then $\overline{\mathcal{M}}(B_{i_j},B_{i_{j+1}})=\mathcal{M}(B_{i_j},B_{i_{j+1}})$ for all $j=1$, \dots,~$n-1$.
Therefore, when $n=2$,
the representing chain has already been chosen.
Assume that representing chains have been chosen for the case $n-1$.
Let $\sum_{\alpha} n_{\alpha} \sigma_{\alpha}$ be a representing chain for
a connected component $P$ of a fibered product of the form
\[
    \dotDelta^q \times_{B_{i_1}} \mathcal{M}(B_{i_1},B_{i_2}) \times_{B_{i_2}} \cdots \times_{B_{i_{n-2}}} \mathcal{M}(B_{i_{n-2}},B_{i_{n-1}}).
\]
For each $\alpha$ let $\Delta^p_{\alpha}=\Delta^p$ denote the domain of $\sigma_{\alpha}$.
Let $\sum_{\beta} m_{\beta} \tau_{\beta}$ be the already chosen representing chain for the fibered product
$\Delta^p_{\alpha} \times_{B_{i_{n-1}}} \mathcal{M}(B_{i_{n-1}},B_{i_{n}})$
of the composition map
\[
    \mathrm{ev}_+\circ\mathrm{pr}\circ\sigma_{\alpha} \colon \Delta^p_{\alpha} \xrightarrow{\sigma_{\alpha}} P \xrightarrow{\mathrm{pr}} \mathcal{M}(B_{i_{n-2}},B_{i_{n-1}}) \xrightarrow{\mathrm{ev}_+} B_{i_{n-1}},
\]
and the beginning map $\mathrm{ev}_- \colon \mathcal{M}(B_{i_{n-1}},B_{i_n}) \to B_{i_{n-1}}$.
Then we define a representing chain for the fibered product
$P \times_{B_{i_{n-1}}} \mathcal{M}(B_{i_{n-1}},B_{i_{n}})$
to be $\sum_{\alpha,\,\beta} n_{\alpha} m_{\beta} \bigl((\sigma_{\alpha} \times \mathrm{id})\circ\tau_{\beta})$.

Now consider the case of $d\geq 2$.
Assume that representing chains have been chosen for the case $d-1$.
We define representing chains for all connected components of the form
$\dotDelta^{p} \times_{B} \overline{\mathcal{M}}(B,B')$ satisfying $\ind{B}-\ind{B'}=d$ by induction on $p$.

For such $B$ and $B'$
the boundary of $\dotDelta^0 \times_B \overline{\mathcal{M}}(B,B')$ can be described in terms of fibered products of the form
$\dotDelta^0 \times_B \overline{\mathcal{M}}(B,B'')\times_{B''} \overline{\mathcal{M}}(B'',B')$.
Here the relative indices $\ind{B}-\ind{B''}$ and $\ind{B''}-\ind{B'}$ are less than or equal to $d-1$.
By the inductive hypothesis, there exists a representing chain for
$\dotDelta^0 \times_B \overline{\mathcal{M}}(B,B'')\times_{B''} \overline{\mathcal{M}}(B'',B')$,
which provides the relative fundamental class.
We extend this representing chain and obtain a representing chain for $\dotDelta^0 \times_B \overline{\mathcal{M}}(B,B')$.

Assume that a representing chain has been chosen for every connected component of the form
$\dotDelta^{p-1} \times_B \overline{\mathcal{M}}(B,B')$ with $\ind{B}-\ind{B'}=d$.
Since the boundary of $\dotDelta^p \times_B \overline{\mathcal{M}}(B,B')$ can be described in terms of fibered products of the forms
$\dotDelta^{p} \times_B \overline{\mathcal{M}}(B,B'')\times_{B''} \overline{\mathcal{M}}(B'',B')$
and
$\dotDelta^{p-1} \times_B \overline{\mathcal{M}}(B,B')$,
we can use the inductive hypotheses on $d$ and $p$.
Now we extend the representing chain of the boundary to obtain a representing chain for $\dotDelta^p \times _{B} \overline{\mathcal{M}}(B,B')$.
Hence in the case of $n=2$,
we have defined a representing chain for $P$.

For $n\geq 3$,
by a similar argument employed in the case of $d=1$ one can choose a suitable representing chain for $P$.
\end{proof}

\begin{remark}
According to Definition \ref{definition:good rep.sys.}, for $P\in C^f_p$,
its representing chain $s_{P}$ always represents the positive (relative) fundamental class of $P$.
However, no problem occurs even if $s_{P}$ always represents the negative (relative) fundamental class in Definition \ref{definition:good rep.sys.}.
For instance, when considering representatives of the positive fundamental class,
in the particular case where $P$ has dimension zero,
the constant map chosen as the representing chain carries the same sign as the orientation assigned to $P$.
On the other hand, if one considers representatives of the negative fundamental class,
the map would carry the opposite sign.
\end{remark}


\subsection{The Morse--Bott--Smale chain complex}

Let $\{C^f_p\}_{p\geq 0}$ be the collection introduced in Section \ref{subsection:simplicial MBS chain complex}.
We fix a good representing chain system $\mathcal{R}$ for $f$.
We define our Morse--Bott--Smale chain complex $(\CB_*(f,\mathcal{R}),\bm{\partial}_*)$ as a quotient of
the larger chain complex $(\widetilde{\CB}_*(f),\bm{\partial}_*)$.

\begin{definition}[Degeneracy condition]\label{definition:deg con}
Let $i=0$,~1, \dots,~$\dim{M}$ and $p\geq 0$.
We define the subgroup $\mathcal{D}_p(B_i)\subset \mathcal{S}_p(B_i)$
of \textit{degenerate} singular topological chains to be the subgroup generated by the following elements.
\begin{itemize}
    \item Let $\sigma_{P}$ be a smooth singular $C_p$-space in $B_i$.
    If $s_P=\sum_{\alpha} n_{\alpha}\sigma_{\alpha}\in\mathcal{R}$ is the representing chain for $P$,
    then $\sigma_{P}-\sum_{\alpha} n_{\alpha}(\sigma_P\circ\sigma_{\alpha})\in \mathcal{D}_p(B_i)$.
\end{itemize}
\end{definition}

\begin{remark}\label{remark:private_communication}
In order to clarify the differences between our construction and that of Banyaga and Hurtubise, we first recall their definition.
Originally, Banyaga and Hurtubise introduced the following five degeneracy relations \cite[Definition~5.9]{BH10} to define their Morse--Bott--Smale chain complex.
We remark that they worked with faces of the $N$-dimensional cube $I^N$ rather than faces of the standard $p$-simplex $\dotDelta$.
In their notation, the symbol $D^{\infty}_p$ plays the role of our $\mathcal{D}_p$.

\begin{enumerate}[label=(\arabic{enumi})]
    \item If $\alpha$ is an orientation preserving homeomorphism such that $\sigma_Q\circ\alpha=\sigma_P$
    and $\partial_0\sigma_Q\circ\alpha=\partial_0\sigma_P$,
    then $\sigma_P-\sigma_Q\in D^{\infty}_p(B_i)$.
    \item If $P$ is a face of $I^N$ and $\sigma_P$ does not depend on some free coordinate of $P$,
    then $\sigma_P\in D^{\infty}_p(B_i)$ and $\partial_j(\sigma_P)\in D^{\infty}_{p+j-1}(B_{i-j})$
    for all $j=1$, \dots,~$m$.
    \item If $P$ and $Q$ are connected components of some fibered products
    and $\alpha$ is an orientation reversing map such that $\sigma_Q\circ\alpha=\sigma_P$
    and $\partial_0\sigma_Q\circ\alpha=\partial_0\sigma_P$,
    then $\sigma_P+\sigma_Q\in D^{\infty}_p(B_i)$.
    \item If $Q$ is a face of $I^N$ and $R$ is a connected component of a fibered product
    \[
    Q%
    \times_{B_{i_1}} \overline{\mathcal{M}}(B_{i_1},B_{i_2})%
    \times_{B_{i_2}} \overline{\mathcal{M}}(B_{i_2},B_{i_3})%
    \times_{B_{i_3}} \dots%
    \times_{B_{i_{n-1}}} \overline{\mathcal{M}}(B_{i_{n-1}},B_{i_n})
    \]
    such that $\deg{R}>\dim{B_{i_n}}$,
    then $\sigma_R\in D^{\infty}_r(B_{i_n})$ and $\partial_j(\sigma_R)\in D^{\infty}_{r+j-1}(B_{i_n-j})$
    for all $j=0$,~$1$, \dots,~$m$.
    \item If $\sum_{\alpha} n_{\alpha}\sigma_{\alpha}\in S_*(R)$
    is a smooth singular chain in a connected component $R$ of a fibered product (as in (4))
    that represents the fundamental class of $R$ and
    \[
        (-1)^{r+i_n}\partial_0\sigma_R - \sum_{\alpha} n_{\alpha}\partial(\sigma_R\circ\sigma_{\alpha}) \in D^{\infty}_{r-1}(B_{i_n}),
    \]
    then
    \[
        \sigma_R - \sum_{\alpha} n_{\alpha}(\sigma_R\circ\sigma_{\alpha}) \in D^{\infty}_r(B_{i_n})
    \]
    and
    \[
        \partial_j\left(\sigma_R - \sum_{\alpha} n_{\alpha}(\sigma_R\circ\sigma_{\alpha})\right) \in D^{\infty}_{r+j-1}(B_{i_n-j})
    \]
    for all $j=1$, \dots,~$\dim{M}$.
\end{enumerate}

In private correspondence with the authors \cite{BH25}, we confirmed that these five relations are in fact redundant, and we record this here with their kind permission.
More precisely, the third and fourth degeneracy relations are not needed.
The fourth relation was included to make explicit that their chain complex reduces to the usual Morse--Smale--Witten complex when the function is Morse--Smale.
More importantly, the justification of the well-definedness of the induced homomorphism between quotients \cite[Lemma~5.10]{BH10} appears to be incomplete, since one needs to verify that the iterated maps $\partial_j$ send degenerate chains to degenerate chains.
A similar difficulty arises in defining the chain map \cite[Lemma~6.10]{BH10} and in establishing independence of the choice of Morse--Bott--Smale functions.

These observations motivate us to revisit the degeneracy conditions and to propose a simplified formulation.
This refinement constitutes one of the main contributions of the present paper, and it serves as the foundation for our construction of the Morse--Bott--Smale chain complex.
\end{remark}

We define the quotient $\mathcal{C}_p(B_i)=\mathcal{S}_p(B_i)/\mathcal{D}_p(B_i)$.

\begin{lemma}\label{lemma:induced hom}
Let $i=0$,~$1$, \dots,~$\dim{M}$ and $p\geq 0$.
For any $j=0$,~$1$, \dots,~$\dim{M}$ 
the homomorphism $\partial_{[j]}\colon \mathcal{S}_p(B_i)\to \mathcal{S}_{p+j-1}(B_{i-j})$ defined in Section \ref{subsection:simplicial MBS chain complex} descends to a homomorphism $\partial_{[j]}\colon \mathcal{C}_p(B_i)\to \mathcal{C}_{p+j-1}(B_{i-j})$.
\end{lemma}

\begin{proof}
Let $\sigma_{P}$ be a smooth singular $C_p$-space in $B_i$.
Let $s_P=\sum_{\alpha} n_{\alpha}\sigma_{\alpha}\in\mathcal{R}$ be the representing chain for $P$.
Then we have $\sigma_P-\sum_{\alpha} n_{\alpha}(\sigma_P\circ\sigma_{\alpha})\in\mathcal{D}_p(B_i)$.
It is enough to show that
\[
    \partial_{[j]}\left(\sigma_P-\sum_{\alpha} n_{\alpha}(\sigma_P\circ\sigma_{\alpha})\right) \in \mathcal{D}_{p+j-1} (B_{i-j})
\]
for all $j=0$,~1, \dots,~$\dim{M}$.

First we consider the case of $j=0$.
The domains of the maps appearing in the computation of
$\partial_{[0]}(\sigma_P)=(-1)^{p+i}\partial_p(\sigma_P)$ correspond to the connected components of the boundary $\partial P$.
These domains carry the signs determined by the orientation of $\partial P$ induced by $P$.
By Definition \ref{definition:rep.sys.},
applying the boundary operator $\partial_p$ to $\sum_{\alpha} n_{\alpha}\sigma_{\alpha}$
matches the representing chains for the connected components of $\partial P$.
Hence the statement holds for $j=0$.

Let $j=1$,~2, \dots,~$\dim{M}$.
We recall that the domains of the maps $\partial_{[j]}(\sigma_P)$ and $\partial_{[j]}(\sigma_P\circ\sigma_{\alpha})$
are connected components of the fibered products
$P\times_{B_i}\overline{\mathcal{M}}(B_i,B_{i-j})$ and $\Delta^p_{\alpha}\times_{B_i}\overline{\mathcal{M}}(B_i,B_{i-j})$,
respectively,
where $\Delta_{\alpha}^{p}=\Delta^{p}$ is the domain of $\sigma_{\alpha}$.
For each $\alpha$ let $\sum_{\beta} m_{\beta}\tau_{\beta}\in\mathcal{R}$
be the representing chain for $\Delta^p _{\alpha} \times_{B_{i}} \overline{\mathcal{M}}(B_{i},B_{i-j})$.
Then the representing chain for $P\times_{B_{i}}\overline{\mathcal{M}}(B_{i},B_{i-j})$ in $\mathcal{R}$ is
$\sum_{\alpha,\,\beta} n_{\alpha} m_{\beta}\bigl((\sigma_{\alpha}\times\mathrm{id})\circ\tau_{\beta}\bigr)$
by Definition \ref{definition:good rep.sys.}.
Therefore,
\[
    \partial_{[j]}(\sigma_P)%
    - \sum_{\alpha,\,\beta} n_{\alpha} m_{\beta} \left(\partial_{[j]}(\sigma_P)\circ(\sigma_{\alpha}\times\mathrm{id})\circ\tau_{\beta}\right)%
    \in\mathcal{D}_{p+j-1} (B_{i-j}).
\]
Moreover, the equation
\[
    \sum_{\alpha,\,\beta} n_{\alpha} m_{\beta} \left(\partial_{[j]}(\sigma_P\circ\sigma_{\alpha})\circ\tau_{\beta}\right)
    =\sum_{\alpha,\,\beta} n_{\alpha} m_{\beta} \left(\partial_{[j]}(\sigma_P)\circ(\sigma_{\alpha}\times\mathrm{id})\circ\tau_{\beta}\right)
\]
holds as singular simplicial chains.
Hence we have
\begin{align*}
    &\partial_{[j]}\left(\sigma_P-\sum_{\alpha} n_{\alpha}(\sigma_P\circ\sigma_{\alpha})\right) \\
    &= \left(\partial_{[j]}(\sigma_P)%
    - \sum_{\alpha,\,\beta} n_{\alpha} m_{\beta} \left(\partial_{[j]}(\sigma_P\circ\sigma_{\alpha})\circ\tau_{\beta}\right)\right)\\
    &\qquad + \left(\sum_{\alpha,\,\beta} n_{\alpha} m_{\beta} \left(\partial_{[j]}(\sigma_P\circ\sigma_{\alpha})\circ\tau_{\beta}\right)%
    - \sum_{\alpha} n_{\alpha} \partial_{[j]}(\sigma_P\circ\sigma_{\alpha})\right) \\
    & = \left(\partial_{[j]}(\sigma_P)%
    - \sum_{\alpha,\,\beta} n_{\alpha} m_{\beta} \left(\partial_{[j]}(\sigma_P)\circ(\sigma_{\alpha}\times\mathrm{id})\circ\tau_{\beta}\right)\right)\\
    &\qquad - \sum_{\alpha} n_{\alpha} \left(\partial_{[j]}(\sigma_P\circ\sigma_{\alpha})%
    - \sum_{\beta} m_{\beta}\left(\partial_{[j]}(\sigma_P\circ\sigma_{\alpha})\circ\tau_{\beta}\right)\right)%
    \in \mathcal{D}_{p+j-1}(B_{i-j}),
\end{align*}
which completes the proof.
\end{proof}

For each $k$ we define
\[
    \CB_k(f,\mathcal{R}) = \bigoplus_{i=0}^{\dim{M}} \mathcal{C}_{k-i}(B_i) = \bigoplus_{i=0}^{\dim{M}} \mathcal{S}_{k-i}(B_i) / \mathcal{D}_{k-i}(B_i).
\]
and define a homomorphism $\bm{\partial}_k\colon\CB_k(f,\mathcal{R})\to\CB_{k-1}(f,\mathcal{R})$ as in \eqref{eq:boundary_op}.
By Proposition \ref{proposition:boundary operator},
we have $\bm{\partial}_{k-1}\circ\bm{\partial}_k=0$ for any $k$.

\begin{definition}
We call the chain complex $(\CB_*(f,\mathcal{R}),\bm{\partial}_*)$ and its homology $\HB_*(f,\mathcal{R})$
the \textit{Morse--Bott--Smale chain complex} and the \textit{Morse--Bott homology} for $(f,\mathcal{R})$, respectively.
\end{definition}

\begin{theorem}\label{theorem:MB homology theorem}
The Morse--Bott homology $\HB_*(f,\mathcal{R})$ does not depend on the choices of the Morse--Bott--Smale function $f\colon M\to\RR$
nor the good representing chain system $\mathcal{R}$.
\end{theorem}

In Section \ref{section:independence},
we provide a proof of Theorem \ref{theorem:MB homology theorem}.


\section{Examples}\label{section:computation}

In this section, we provide several examples of computations.
Let $(M,g)$ be a closed oriented Riemannian manifold.
Let $f\colon M\to\RR$ be a Morse--Bott--Smale function on $(M,g)$.
We assume that all the critical submanifolds of $f$ and their negative normal bundles are oriented.
Let $N>\dim{M}$.

Throughout all the examples in this section,
we fix a good representing chain system $\mathcal{R}$ for $f$, which consists of relative cycles
that represent the positive relative fundamental class.
In particular, when a set $P$ has dimension zero,
the constant map with the same sign as the orientation on $P$ is chosen as its representing chain.


\begin{example}[A constant function]\label{example:Constant func.}
Assume that $f$ is constant.
Then there is the only one critical submanifold $B_0 = M$.
Thus for any $p\geq 0$ we have $\widetilde{\CB}_p(f) = \mathcal{S}_p(M)$.
Moreover, $\partial_{[0]} = (-1)^p\partial_p$
and $\partial_{[j]} = 0$ for all $j>0$.
In this case,
the good representing chain system for $f$ is trivial
since $\{C^f_p\}_{p\geq 0}$ does not involve the moduli space of gradient flow lines.
Moreover, every singular $C^f_p$-space $\sigma\colon\dotDelta^p\to M$,
where $\dotDelta^p$ is a $p$-face of $\Delta^N$,
is equivalent to the smooth singular $p$-simplex $\sigma\circ\iota\colon\Delta^p\to M$
due to the degeneracy condition (Definition \ref{definition:deg con}).
Here $\iota\colon\Delta^p\to\dotDelta^p\subset\Delta^N$ denotes the standard embedding.
Therefore, for any $p\geq 0$ we deduce that $\mathcal{C}_p (M)$ is isomorphic to
the smooth singular simplicial chain group $\mathcal{S}_p^{\Delta}(M)$ of $M$.
Hence the Morse--Bott--Smale chain complex $(\CB_*(f,\mathcal{R}),\bm{\partial}_*)$
coincides with the smooth singular simplicial chain complex of $M$ up to sign.
This chain complex can be pictured as follows:
\[
    \xymatrix@C=20pt@R=20pt{
    \ddots & \vdots \ar@{}[d]|{\oplus} &\vdots \ar@{}[d]|{\oplus} &\vdots \ar@{}[d]|{\oplus} &\vdots \ar@{}[d]|{\oplus} & \\
    \cdots & 0 \ar@{}[d]|{\oplus}& 0 \ar@{}[d]|{\oplus} & 0 \ar@{}[d]|{\oplus} & 0 \ar@{}[d]|{\oplus} &  \\ 
    \cdots &\mathcal{S}_3^{\mathrm{\Delta}}(M) \ar[r]^-{-\partial_3} \ar@{}[d]|{||} &\mathcal{S}_2^{\mathrm{\Delta}}(M) \ar[r]^-{\partial_2} \ar@{}[d]|{||} & \mathcal{S}_1^{\mathrm{\Delta}}(M) \ar[r]^-{-\partial_1} \ar@{}[d]|{||} &\mathcal{S}_0^{\mathrm{\Delta}}(M) \ar[r]^-{\partial_0} \ar@{}[d]|{||} & 0 \\
    \cdots & \CB_3(f,\mathcal{R}) \ar[r]^-{\bm{\partial}_3} & \CB_2(f,\mathcal{R}) \ar[r]^-{\bm{\partial}_2} & \CB_1(f,\mathcal{R}) \ar[r]^-{\bm{\partial}_1} & \CB_0(f,\mathcal{R}) \ar[r]^-{\bm{\partial}_0} & 0
    }
\]
Therefore, the Morse--Bott homology $\HB_*(f,\mathcal{R})$ is isomorphic to the singular homology $H_*(M;\ZZ)$.
\end{example}

Now Theorem \ref{theorem:MB homology theorem} and Example \ref{example:Constant func.} imply the following.

\begin{corollary}\label{corollary:MB_is_sing}
For any Morse--Bott--Smale function $f\colon M\to\RR$ and any good representing chain system $\mathcal{R}$ for $f$
we have an isomorphism $\HB_*(f,\mathcal{R}) \cong H_*(M;\ZZ)$.
\end{corollary}


\begin{example}\label{example:Morse--Bott func 1}
Assume that $M = S^2 = \{\, (x,y,z) \in \RR ^3 \mid x^2+y^2+z^2=1 \,\}$,
and that $f(x,y,z)=z^2$.
Then $B_0 \approx S^1$, $B_1 = \emptyset$, and $B_2 = \{n,s\}$ where $n=(0,0,1)$ and $s=(0,0,-1)$.

\begin{figure}[ht]
    \centering
    \begin{tikzpicture}
    \draw [thick](0,0) circle[radius=2.5];
    \draw [teal,thick](-2.5,0) to [out=-30,in=-150] (2.5,0);
    \draw [dotted,teal,thick](-2.5,0) to [out=20,in=160] (2.5,0);
    \draw (2.5,0) node[below right,teal,thick]{{\LARGE$B_0$}};
    \draw (2,3) node[below right,thick]{{\LARGE$S^2$}};
    \draw[->][left,magenta,line width=0.5mm] (-3.3,0.5) to[bend left] (-0.5,2.5);
    \draw[->][left,magenta,line width=0.5mm] (-3.3,-0.5) to[bend right] (-0.5,-2.5);
    \draw (-3.5,0) node[magenta,thick]{{\LARGE$B_2$}};
    \draw[->] (4,0) -- (6,0) node [midway,above] {$f$};
    \draw[->][left,line width=0.5mm] (-0.3,2.3) to[bend right] (-1.8,1.2);
    \draw[->][left,line width=0.5mm] (0.1,2.3)-- (0.3,1.2);
    \draw[->][left,line width=0.5mm] (0.4,2.3)to[bend left] (1.7,1.2);
    \draw[->][left,line width=0.5mm] (-0.3,-2.3)to[bend left] (-1.8,-1.2);
    \draw[->][left,line width=0.5mm] (-0.1,-2.2)to[bend left] (-0.6,-1.2);
    \draw[->][left,line width=0.5mm] (0.4,-2.3)to[bend right] (1.7,-1.2);
    \draw [thick](7,-2.5) -- (7,2.5); 
    \draw [thick](6.8,-2.5) -- (7.2,-2.5) node[right]{$1$};
    \draw [thick](6.8,2.5) -- (7.2,2.5) node[right]{$1$};
    \draw [thick](6.8,0) -- (7.2,0) node[right]{$0$};
    \draw[magenta,fill] (0,2.5) circle (2pt) node [above=1mm] {$n$};
    \draw[magenta,fill] (0,-2.5) circle (2pt) node [below=1mm] {$s$};
\end{tikzpicture}
\caption{The critical submanifolds of $f$. The black arrows represent negative gradient flow lines of $f$.}
\label{figure:S^2_z^2_crit}
\end{figure}
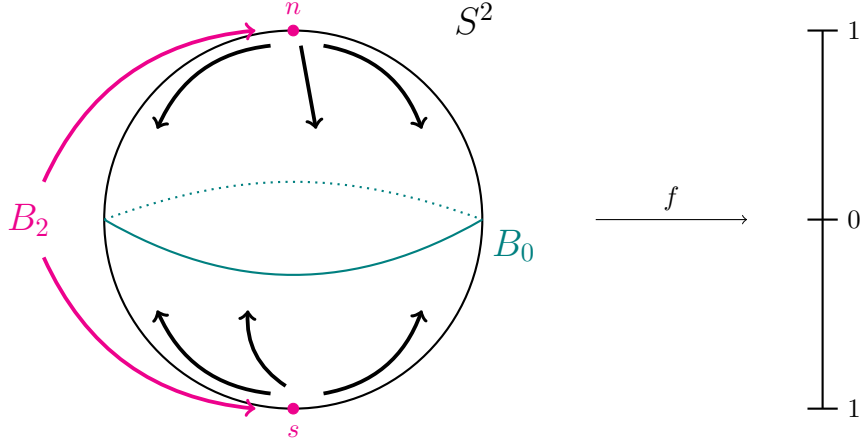

Since $\dim{B_2}=0$,
we have $\mathcal{C}_p(B_2) = \ZZ n \oplus \ZZ s$
for each $p\geq 0$.
Here the constant maps to $n$ and $s$ are identified with the points $n$ and $s$, respectively.
Then the Morse--Bott--Smale chain complex restricted to the row of $B_2$ coincides with the singular simplicial chain complex of $B_2$.
Consequently,
the map $\partial_{[0]} \colon \mathcal{C}_p(B_2)\to \mathcal{C}_{p-1}(B_2)$
is an isomorphism if $p$ is a positive even integer,
and it is the zero map otherwise.
Therefore, the Morse--Bott--Smale chain complex for $f$ is as follows:
\[
    \xymatrix@C=20pt@R=20pt{
    \cdots &\ZZ n \oplus \ZZ s \ar[r]^-{0} \ar@{}[d]|{\oplus} \ar[ddr]^-{\partial_{[2]}}& \ZZ n \oplus \ZZ s \ar[r]^-{\partial_{[0]}} \ar[ddr]^-{\partial_{[2]}} \ar@{}[d]|{\oplus}  & 0 \ar@{}[d]|{\oplus} & 0 \ar@{}[d]|{\oplus}& \\
    \cdots & 0  \ar@{}[d]|{\oplus}& 0 \ar@{}[d]|{\oplus} & 0  \ar@{}[d]|{\oplus} & 0 \ar@{}[d]|{\oplus} &  \\ 
    \cdots &\mathcal{C}_3 (B_0) \ar[r]^{\partial_{[0]}} \ar@{}[d]|{||} & \mathcal{C}_2 (B_0) \ar[r]^{\partial_{[0]}} \ar@{}[d]|{||} & \mathcal{C}_1 (B_0) \ar[r]^{\partial_{[0]}} \ar@{}[d]|{||} & \mathcal{C}_0 (B_0) \ar[r]^-{\partial_{[0]}} \ar@{}[d]|{||} & 0 \\
    \cdots & \CB_3(f,\mathcal{R}) \ar[r]^{\bm{\partial}_3} & \CB_2(f,\mathcal{R}) \ar[r]^{\bm{\partial}_2} & \CB_1(f,\mathcal{R}) \ar[r]^-{\bm{\partial}_1} & \CB_0(f,\mathcal{R}) \ar[r]^-{\bm{\partial}_0} & 0
    }
\]
It is clear that $\HB_0(f,\mathcal{R})\cong\ZZ$ since the row of $B_0$ is the singular simplicial chain complex of $B_0\approx S^1$ up to sign.

We show that $\HB_k(f,\mathcal{R})=0$ for all $k\geq 3$.
Let $(c,d) \in \Ker{\bm{\partial}_k} \subset \mathcal{C}_{k-2}(B_2) \oplus \mathcal{C}_k(B_0)$.
In particular, $c \in \Ker \bigl(\partial_{[0]} \colon \mathcal{C}_{k-2}(B_2)\to \mathcal{C}_{k-3}(B_2)\bigr)$.
Since $H_{k-2}(B_{2})=0$ whenever $k\geq 3$,
there exists $c'\in C_{k-1}(B_2)$ such that $\partial_{[0]}c'=c$.
Then we have $(c,\partial_{[2]}c') = \bm{\partial}_{k+1}(c',0) \in \Image{\bm{\partial}_{k+1}}$.
Therefore, $[(c,d)]=[(0,d-\partial_{[2]}c')]\in \HB_k(f,\mathcal{R})$.
Since $[d-\partial_{[2]}c']\in H_{k}(B_0)=0$,
we conclude that $[(c,d)]=0$.
 
We compute $\HB_1(f,\mathcal{R})$.
Consider the map $\partial_{[2]} \colon \ZZ n \oplus \ZZ s \to \mathcal{C}_1(B_0)$.
By assigning appropriate orientations to $M$, the critical submanifolds, and their negative normal bundles,
we assume that the moduli space $\overline{\mathcal{M}}(B_2,B_0)$ is oriented as shown in Figure \ref{figure:S^2_z^2_moduli}.

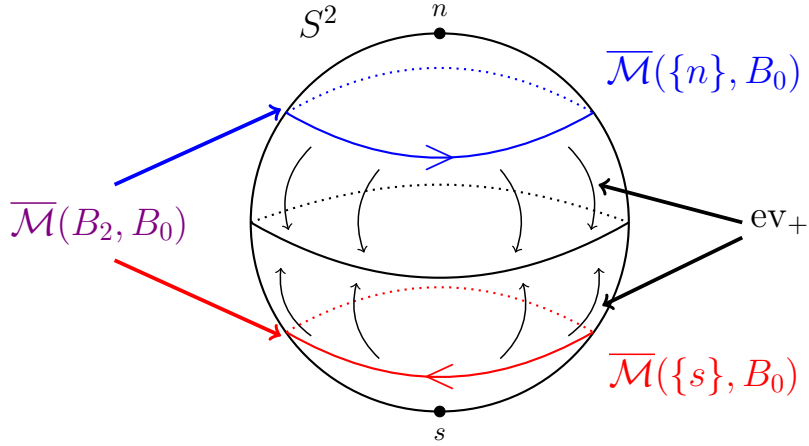
\begin{figure}[ht]
\centering
\begin{tikzpicture}
    \draw [thick](0,0) circle[radius=2.5];
    \draw [thick](-2.5,0) to [out=-30,in=-150] (2.5,0);
    \draw [thick,blue,](-2.03,1.45) to [out=-30,in=-150]node{{\huge$>$}} (2.03,1.45);
    \draw [thick,dotted,blue](-2.03,1.45) to [out=30,in=150] (2.03,1.45);
    \draw [thick,red](-2.03,-1.45) to [out=-30,in=-150]node{{\huge$<$}} (2.03,-1.45);
    \draw [thick,dotted,red](-2.03,-1.45) to [out=30,in=150] (2.03,-1.45);
    \draw [dotted,thick](-2.5,0) to [out=20,in=160] (2.5,0);
    \draw (-4.5,0) node[red!50!blue,thick]{{\LARGE$\overline{\mathcal{M}}(B_{2},B_{0})$}};
    \draw (3.5,2) node[blue,thick]{{\LARGE$\overline{\mathcal{M}}(\{n\},B_{0})$}};
    \draw (3.5,-2) node[red,thick]{{\LARGE$\overline{\mathcal{M}}(\{s\},B_{0})$}};
    \draw (4.5,0) node{{\LARGE$\mathrm{ev}_+$}};
    \draw[->][line width=0.5mm] (4,0) -- (2.1,0.5);
    \draw[->][line width=0.5mm] (4,-0.2) -- (2.15,-1.1);
    \draw[->][line width=0.2mm](-1.7,1)to[bend right] (-2,-0.1);
    \draw[->][line width=0.2mm](-0.8,0.7)to[bend right] (-1,-0.4);
    \draw[->][line width=0.2mm](1.7,1)to[bend left] (2,-0.1);
    \draw[->][line width=0.2mm](0.8,0.7)to[bend left] (1,-0.4);
    \draw[->][line width=0.2mm](-1.7,-1.5)to[bend left] (-2.1,-0.6);
    \draw[->][line width=0.2mm](-0.8,-1.8)to[bend left] (-1.1,-0.8);
    \draw[->][line width=0.2mm](1.7,-1.5)to[bend right] (2.1,-0.6);
    \draw[->][line width=0.2mm](0.8,-1.8)to[bend right] (1.1,-0.8);
    \draw[->][blue,line width=0.5mm] (-4.3,0.5) -- (-2.1,1.5);
    \draw[->][red,line width=0.5mm] (-4.3,-0.5) -- (-2.1,-1.5);
    \draw (-2,3) node[below right,thick]{{\LARGE$S^2$}};
    \draw[fill] (0,2.5) circle (2pt) node [above=1mm] {$n$};
    \draw[fill] (0,-2.5) circle (2pt) node [below=1mm] {$s$};
\end{tikzpicture}
\caption{The moduli space $\overline{\mathcal{M}}(B_{2},B_{0})$.}
\label{figure:S^2_z^2_moduli}
\end{figure}

Let $\sigma_n \colon \Delta^0 \to B_2$ and $\sigma_s\colon \Delta ^0 \to B_2$
denote the constant maps taking values $n$ and $s$, respectively.
These maps represent the generators of $\mathcal{C}_0(B_2)$.
The map $\partial_{[2]}(\sigma_n)$ is defined to be the composition map
\[
        \mathrm{ev}_+ \circ \mathrm{pr}_2%
        \colon \Delta^0 \times_{B_2} \overline{\mathcal{M}}(B_2,B_0)%
        \xrightarrow{\mathrm{pr}_2} \overline{\mathcal{M}}(B_2,B_0)%
        \xrightarrow{\mathrm{ev}_+} B_0,
\]
where the fibered product is taken with respect to $\sigma_n$ and the beginning map $\mathrm{ev}_-$
The same applies to $\partial_{[2]}(\sigma_s)$ as well.
Let $\sum_{\alpha} n_{\alpha} \sigma_{\alpha}$ and $\sum _{\beta} m_{\beta}\sigma _{\beta} \in \mathcal{R}$ be the representing chains for the domains of $\partial_{[2]}(\sigma_n)$ and $\partial_{[2]}(\sigma_s)$, respectively.
By definition \ref{definition:deg con},
we have
\[
    \partial_{[2]}(\sigma_n)=\sum_{\alpha} n_{\alpha}(\mathrm{ev}_+\circ \mathrm{pr}_2\circ\sigma_{\alpha})%
    \quad \text{and}\quad \partial_{[2]}(\sigma_s)=\sum_{\beta} m_{\beta}(\mathrm{ev}_+\circ \mathrm{pr}_2\circ\sigma_{\beta})
\]
as elements of $\mathcal{C}_1(B_0)=\mathcal{S}_1(B_0)/\mathcal{D}_1(B_0)$.
Moreover, they are chains that traverse $B_{0}$ in accordance with the orientation induced by their respective moduli spaces.
Thus, $\bm{\partial}_2(\sigma_n,0)=\partial_{[2]}(\sigma_n)$ and $\bm{\partial}_2(\sigma_s,0)=\partial_{[2]}(\sigma_s)$ correspond to generators of $H_1(B_0)\cong \ZZ$,
which implies that $\HB_1(f,\mathcal{R})=0$.

On the other hand, the sum $\partial_{[2]}(\sigma_n+\sigma_s)$ does not correspond to a generator of $H_1(B_0)$.
Then there exists $e\in \mathcal{C}_2(B_0)$ such that $\partial_{[0]}(e)=\partial_{[2]}(\sigma_n+\sigma_s)$.
Thus $(\sigma_n+\sigma_s,-e)\in\Ker\bm{\partial}_2$.
Since $H_2(B_0)=0$ and $\partial_{[0]}\colon \mathcal{C}_1(B_2) \to \mathcal{C}_0(B_2)$ is zero,
we conclude that $\HB_2(f,\mathcal{R})\cong \ZZ$.
\end{example}


\begin{example} \label{example:Morse--Bott func 2}
Assume that $M=S^2$ and $f(x,y,z)=-z^2$.
Then $B_0=\{n,s\}$, $B_1 \approx S^1$, and $B_2=\emptyset$.

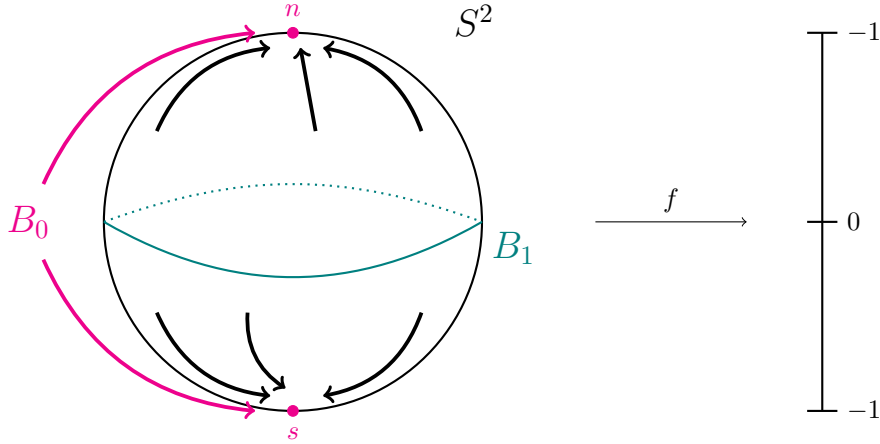
\begin{figure}[ht]
    \centering
    \begin{tikzpicture}
    \draw [thick](0,0) circle[radius=2.5];
    \draw [teal,thick](-2.5,0) to [out=-30,in=-150] (2.5,0);
    \draw [dotted,teal,thick](-2.5,0) to [out=20,in=160] (2.5,0);
    \draw (2.5,0) node[below right,teal,thick]{{\LARGE$B_1$}};
    \draw (2,3) node[below right,thick]{{\LARGE$S^2$}};
    \draw[->][left,magenta,line width=0.5mm] (-3.3,0.5) to[bend left] (-0.5,2.5);
    \draw[->][left,magenta,line width=0.5mm] (-3.3,-0.5) to[bend right] (-0.5,-2.5);
    \draw (-3.5,0) node[magenta,thick]{{\LARGE$B_0$}};
    \draw[->] (4,0) -- (6,0) node [midway,above] {$f$};
    \draw[->][left,line width=0.5mm] (-1.8,1.2) to[bend left] (-0.3,2.3);
    \draw[->][left,line width=0.5mm] (0.3,1.2)-- (0.1,2.3);
    \draw[->][left,line width=0.5mm] (1.7,1.2)to[bend right] (0.4,2.3);
    \draw[->][left,line width=0.5mm] (-1.8,-1.2)to[bend right] (-0.3,-2.3);
    \draw[->][left,line width=0.5mm] (-0.6,-1.2)to[bend right] (-0.1,-2.2);
    \draw[->][left,line width=0.5mm] (1.7,-1.2)to[bend left] (0.4,-2.3);
    \draw [thick](7,-2.5) -- (7,2.5); 
    \draw [thick](6.8,-2.5) -- (7.2,-2.5) node[right]{$-1$};
    \draw [thick](6.8,2.5) -- (7.2,2.5) node[right]{$-1$};
    \draw [thick](6.8,0) -- (7.2,0) node[right]{$0$};
    \draw[magenta,fill] (0,2.5) circle (2pt) node [above=1mm] {$n$};
    \draw[magenta,fill] (0,-2.5) circle (2pt) node [below=1mm] {$s$};
\end{tikzpicture}
\caption{The critical submanifolds of $f$. The black arrows represent negative gradient flow lines of $f$.}
\label{figure:S^2_-z^2_crit}
\end{figure}

Let $\sigma_n \colon \Delta^0 \to B_0$ and $\sigma_s\colon \Delta ^0 \to B_0$
denote the constant maps taking values $n$ and $s$, respectively.
We then have $\mathcal{C}_0(B_0)=\ZZ\sigma_n \oplus \ZZ\sigma_s$
and the map $\partial_{[0]} \colon \mathcal{C}_p(B_2)\to \mathcal{C}_{p-1}(B_2)$ is an isomorphism when $p$ is a positive even integer,
and it is the zero map otherwise.
Therefore, the Morse--Bott--Smale chain complex for $f$ is as follows:
\[
    \xymatrix@C=20pt@R=20pt{
    \cdots & 0 \ar@{}[d]|{\oplus} &  0 \ar@{}[d]|{\oplus}  & 0 \ar@{}[d]|{\oplus} & 0 \ar@{}[d]|{\oplus}& \\
    \cdots & \mathcal{C}_2(B_1) \ar[r]^-{\partial_{[0]}} \ar[dr]^-{\partial_{[1]}} \ar@{}[d]|{\oplus}& \mathcal{C}_1(B_1)\ar[r]^-{\partial_{[0]}} \ar[dr]^-{\partial_{[1]}} \ar@{}[d]|{\oplus} &\mathcal{C}_0(B_1) \ar[r]^-{\partial_{[0]}} \ar[dr]^-{\partial_{[1]}} \ar@{}[d]|{\oplus} & 0 \ar@{}[d]|{\oplus} & \\ 
    \cdots &\mathcal{C}_3(B_0) \ar[r]^-{0} \ar@{}[d]|{||} & \mathcal{C}_2(B_0) \ar[r]^-{\cong} \ar@{}[d]|{||} & \mathcal{C}_1(B_0) \ar[r]^-{0} \ar@{}[d]|{||} & \ZZ\sigma_n \oplus \ZZ\sigma_s \ar[r]^-{\partial_{[0]}} \ar@{}[d]|{||} & 0 \\
    \cdots & \CB_3(f,\mathcal{R}) \ar[r]^-{\bm{\partial}_3} & \CB_2(f,\mathcal{R}) \ar[r]^-{\bm{\partial}_2} & \CB_1(f,\mathcal{R}) \ar[r]^-{\bm{\partial}_1} & \CB_0(f,\mathcal{R}) \ar[r]^-{\bm{\partial}_0} & 0
    }
\]
A similar argument to Example \ref{example:Morse--Bott func 1} shows that
$\HB_k(f,\mathcal{R})=0$ for any $k\geq 3$.

Let $(c,d) \in \Ker{\bm{\partial}_2} \subset \mathcal{C}_1(B_1) \oplus \mathcal{C}_2(B_0)$.
Then $\partial_{[1]}(c)+\partial_{[0]}(d)=0$.
Since $\partial_{[0]}\colon\mathcal{C}_2(B_0)\to\mathcal{C}_1(B_0)$ is an isomorphism,
we have $d=-\partial_{[0]}^{-1}\bigl(\partial_{[1]}(c)\bigr)$.
Therefore, if $[c]$ generates $H_{1}(B_{1})\cong \ZZ$, then
$[(c,d)]$ generates $\HB_2(f,\mathcal{R})$.
Hence $\HB_2(f,\mathcal{R}) \cong \ZZ$.

Next we focus on the map $\partial_{[1]}\colon\mathcal{C}_0(B_1)\to\mathcal{C}_0(B_0)$.
Let $\sigma\colon \Delta^0\to B_1$ be a singular $C^f_0$-space in $B_1$.
By assigning appropriate orientations to $M$, the critical submanifolds, and their negative normal bundles,
we assume that the moduli space $\overline{\mathcal{M}}(B_1,B_0)$ is oriented as shown in Figure \ref{figure:S^2_-z^2_moduli}.

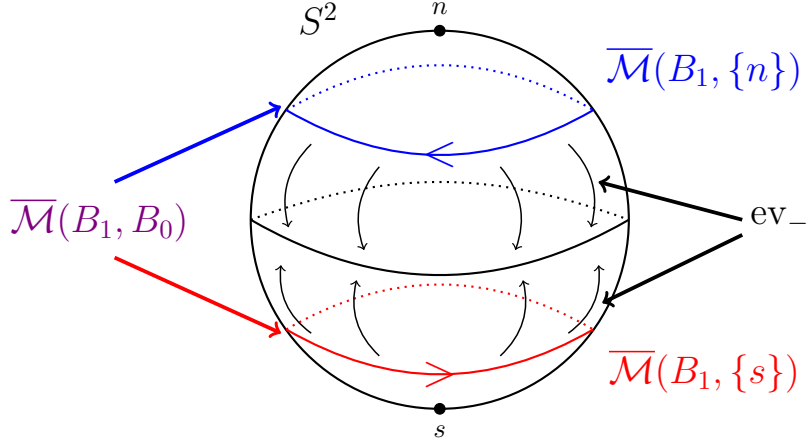
\begin{figure}[ht]
\centering
\begin{tikzpicture}
    \draw [thick](0,0) circle[radius=2.5];
    \draw [thick](-2.5,0) to [out=-30,in=-150] (2.5,0);
    \draw [thick,blue](-2.03,1.45) to [out=-30,in=-150]node{{\huge$<$}} (2.03,1.45);
    \draw [thick,dotted,blue](-2.03,1.45) to [out=30,in=150] (2.03,1.45);
    \draw [thick,red](-2.03,-1.45) to [out=-30,in=-150]node{{\huge$>$}} (2.03,-1.45);
    \draw [thick,dotted,red](-2.03,-1.45) to [out=30,in=150] (2.03,-1.45);
    \draw [dotted,thick](-2.5,0) to [out=20,in=160] (2.5,0);
    \draw (-4.5,0) node[red!50!blue,thick]{{\LARGE$\overline{\mathcal{M}}(B_{1},B_{0})$}};
    \draw (3.5,2) node[blue,thick]{{\LARGE$\overline{\mathcal{M}}(B_{1},\{n\})$}};
    \draw (3.5,-2) node[red,thick]{{\LARGE$\overline{\mathcal{M}}(B_{1},\{s\})$}};
    \draw (4.5,0) node{{\LARGE$\mathrm{ev}_-$}};
    \draw[->][line width=0.5mm] (4,0) -- (2.1,0.5);
    \draw[->][line width=0.5mm] (4,-0.2) -- (2.15,-1.1);
    \draw[->][line width=0.2mm](-1.7,1)to[bend right] (-2,-0.1);
    \draw[->][line width=0.2mm](-0.8,0.7)to[bend right] (-1,-0.4);
    \draw[->][line width=0.2mm](1.7,1)to[bend left] (2,-0.1);
    \draw[->][line width=0.2mm](0.8,0.7)to[bend left] (1,-0.4);
    \draw[->][line width=0.2mm](-1.7,-1.5)to[bend left] (-2.1,-0.6);
    \draw[->][line width=0.2mm](-0.8,-1.8)to[bend left] (-1.1,-0.8);
    \draw[->][line width=0.2mm](1.7,-1.5)to[bend right] (2.1,-0.6);
    \draw[->][line width=0.2mm](0.8,-1.8)to[bend right] (1.1,-0.8);
     \draw[->][blue,line width=0.5mm] (-4.3,0.5) -- (-2.1,1.5);
    \draw[->][red,line width=0.5mm] (-4.3,-0.5) -- (-2.1,-1.5);
    \draw (-2,3) node[below right,thick]{{\LARGE$S^2$}};
    \draw[fill] (0,2.5) circle (2pt) node [above=1mm] {$n$};
    \draw[fill] (0,-2.5) circle (2pt) node [below=1mm] {$s$};
\end{tikzpicture}
\caption{The moduli space $\overline{\mathcal{M}}(B_{1},B_{0})$.}
\label{figure:S^2_-z^2_moduli}
\end{figure}

Consider the fibered product $\Delta^0 \times_{B_1} \overline{\mathcal{M}}(B_1,B_0)$
of the map $\sigma$ and the beginning map $\mathrm{ev}_-\colon \overline{\mathcal{M}}(B_1,B_0) \to B_1$.
It consists of two connected component $\Delta^0 \times_{B_1} \overline{\mathcal{M}}(B_1,\{n\})$
and $\Delta^0 \times_{B_1} \overline{\mathcal{M}}(B_1,\{s\})$.
We note that the composition map
\[
    \mathrm{ev}_+\circ\mathrm{pr}_2%
    \colon \Delta^0 \times_{B_1} \overline{\mathcal{M}}(B_{1},\{n\})%
    \xrightarrow{\mathrm{pr}_2} \overline{\mathcal{M}}(B_{1},\{n\})%
    \xrightarrow{\mathrm{ev}_+} \{n\}
\]
can be identified with the mapping $\sigma_n$ by the degeneracy condition (Definition \ref{definition:deg con}) up to sign.
The same holds for $\sigma_s$.
The signs are determined by the orientations assigned to each connected component of
$\Delta^0\times_{B_1}\overline{\mathcal{M}}(B_1,B_0)$.
Since the orientations of each component of $\overline{\mathcal{M}}(B_{1},B_{0})$
are opposite to each other (see Figure \ref{figure:S^2_-z^2_moduli}),
we have $\partial_{[1]}(\sigma) = \pm(\sigma_n - \sigma_s)$.
Therefore,
\[
    \HB_0(f,\mathcal{R}) \cong \frac{\ZZ\sigma_n\oplus\ZZ\sigma_s}{\ZZ(\sigma_n-\sigma_s)} \cong \ZZ.
\]

Let $(c,d)\in\Ker{\bm{\partial}_1}\subset\mathcal{C}_0(B_1)\oplus\mathcal{C}_1(B_0)$.
Since $(0,d)\in\Image{\bm{\partial}_2}$, we have $[(c,d)]=[(c,0)]\in\HB_1(f,\mathcal{R})$.
Moreover, $\partial_{[1]}(c)=-\partial_{[0]}(d)=0$.
If $c=\sum_kn_k\sigma_k$, then $0=\partial_{[1]}(c)=\pm\sum_kn_k(\sigma_n-\sigma_s)$
and hence $\sum_kn_k=0$.
Therefore, $c\in\Image\bigl(\partial_{[0]}\colon\mathcal{C}_1(B_1)\to\mathcal{C}_0(B_1)\bigr)$.
It concludes that $\HB_1(f,\mathcal{R})=0$.
\end{example}

\begin{example}\label{example:Perfect Morse--Bott func}
Assume that $M$ is the torus $\mathbb{T}^2$ resting horizontally on the plane $z=0$ in $\RR^3$
and that $f\colon \mathbb{T}^2 \to \RR$ is a height function.
Then $B_0\approx B_1 \approx S^1$ and $B_2=\emptyset$.

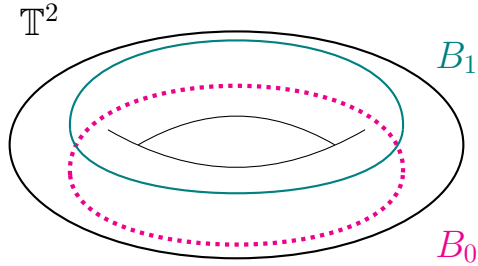
\begin{figure}[ht]
\centering
\begin{tikzpicture}
    \draw[thick](0,0) circle [x radius=3,y radius=1.5,rotate=0];
    \draw[ultra thick,magenta,dotted](-2.2,-0.4) to [out=95,in=85] [distance=45pt](2.2,-0.4);
    \draw[ultra thick,magenta,dotted](-2.2,-0.4) to [out=-85,in=-95] [distance=35pt](2.2,-0.4);
    \draw[thick,teal](-2.2,0.2) to [out=95,in=85] [distance=45pt](2.2,0.2);
    \draw[thick,teal](-2.2,0.2) to [out=-85,in=-95] [distance=32pt](2.2,0.2);
    \draw (-1.7,0.2) to [out=-30,in=-150] (1.7,0.2);
    \draw (-1.3,0) to [out=30,in=150] (1.3,0);
    \draw (2.5,1.5) node[below right,teal,thick]{{\LARGE$B_1$}};
    \draw (2.5,-1) node[below right,magenta,thick]{{\LARGE$B_0$}};
    \draw (-3,2) node[below right,thick]{{\LARGE$\mathbb{T}^2$}};
\end{tikzpicture}
\caption{The critical submanifolds of $f$.}
\label{figure:T^2_crit}
\end{figure}

The Morse--Bott--Smale chain complex is as follows:
\[
    \xymatrix@C=20pt@R=20pt{
    \cdots & 0 \ar@{}[d]|{\oplus} &  0 \ar@{}[d]|{\oplus}  & 0 \ar@{}[d]|{\oplus} & 0 \ar@{}[d]|{\oplus}& \\
    \cdots & \mathcal{C}_2 (B_1) \ar[r]^{\partial_{[0]}}  \ar[dr]^{\partial_{[1]}} \ar@{}[d]|{\oplus}& \mathcal{C}_1 (B_1)\ar[r]^{\partial_{[0]}} \ar[dr]^{\partial_{[1]}} \ar@{}[d]|{\oplus} &\mathcal{C}_0 (B_1) \ar[r]^-{\partial_{[0]}}\ar[dr]^{\partial_{[1]}}  \ar@{}[d]|{\oplus} & 0 \ar@{}[d]|{\oplus} &  \\ 
    \cdots &\mathcal{C}_3 (B_0) \ar[r]^{\partial_{[0]}} \ar@{}[d]|{||} & \mathcal{C}_2 (B_0) \ar[r]^{\partial_{[0]}}  \ar@{}[d]|{||} & \mathcal{C}_1 (B_0) \ar[r]^{\partial_{[0]}} \ar@{}[d]|{||} &\mathcal{C}_0 (B_0) \ar[r]^-{\partial_{[0]}}  \ar@{}[d]|{||} & 0 \\
    \cdots & \CB_3(f,\mathcal{R}) \ar[r]^-{\bm{\partial}_3} & \CB_2(f,\mathcal{R}) \ar[r]^-{\bm{\partial}_2} & \CB_1(f,\mathcal{R}) \ar[r]^-{\bm{\partial}_1} & \CB_0(f,\mathcal{R}) \ar[r]^-{\bm{\partial}_0} & 0
    }
\]
A similar argument to Example \ref{example:Morse--Bott func 1} shows that
$\HB_k(f,\mathcal{R})=0$ for any $k\geq 3$.

By assigning appropriate orientations to $M$, the critical submanifolds, and their negative normal bundles,
we assume that the moduli space $\overline{\mathcal{M}}(B_1,B_0)$ is oriented as shown in Figure \ref{figure:T^2_moduli}.

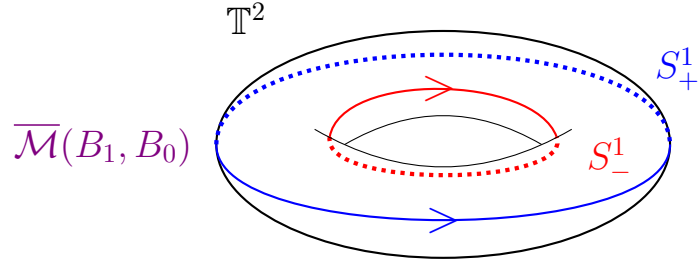
\begin{figure}[ht]
\centering
\begin{tikzpicture}
    \draw[thick](0,0) circle [x radius=3,y radius=1.5,rotate=0];
    \draw[ultra thick,blue,dotted](-3,0) to [out=89,in=91] [distance=45pt](3,0);
    \draw[thick,teal,blue](-3,0) to [out=-91,in=-89] [distance=38pt]node{{\huge$>$}}(3,0);
    \draw[thick,red](-1.5,0.1) to [out=80,in=110] [distance=25pt]node{{\huge$>$}}(1.5,0.1);
    \draw[ultra thick,red,dotted](-1.5,0.1) to [out=-100,in=-70] [distance=20pt](1.5,0.1);
    \draw (-1.7,0.2) to [out=-30,in=-150] (1.7,0.2);
    \draw (-1.3,0) to [out=30,in=150] (1.3,0);
    \draw (-3,2) node[below right,thick]{{\LARGE$\mathbb{T}^2$}};
    \draw (-4.5,0) node[red!50!blue,thick]{{\LARGE$\overline{\mathcal{M}}(B_{1},B_{0})$}};
    \draw (3.1,1) node[blue,thick]{{\LARGE$S^1_+$}};
    \draw (2.2,-0.2) node[red,thick]{{\LARGE$S^1_-$}};
\end{tikzpicture}
\caption{The moduli space $\overline{\mathcal{M}}(B_{1},B_{0})=S^1_+\sqcup S^1_-$.}
\label{figure:T^2_moduli}
\end{figure}

For each $p$ we consider the map $\partial_{[1]} \colon \mathcal{C}_{p}(B_1)\to \mathcal{C}_{p}(B_0)$.
Let $\sigma\colon\Delta^p\to B_1$ be a map.
Consider the fibered product $\Delta^p\times_{B_1}\overline{\mathcal{M}}(B_1,B_0)$
of the maps $\sigma$ and $\mathrm{ev}_-$.
It consists of two connected component $S^1_+$ and $S^1_-$ as in Figure \ref{figure:T^2_moduli}.
We define $\sigma_{\pm}$ to be the composition map
\[
    \mathrm{ev}_+\circ\mathrm{pr}_2%
    \colon S^1_{\pm}%
    \hookrightarrow \Delta^p\times_{B_1}\overline{\mathcal{M}}(B_1,B_0)%
    \xrightarrow{\mathrm{pr}_2} \overline{\mathcal{M}}(B_1,B_0)%
    \xrightarrow{\mathrm{ev}_+} B_0.
\]
Then we have $\partial_{[1]}(\sigma) = \sigma_+ + \sigma_-$.

We note that the connected components $S^1_+$ and $S^1_-$ have opposite orientations.
In particular, for the case $p=0$
we have $\sigma_+ + \sigma_- = 0$ from the degeneracy condition (Definition \ref{definition:deg con}).
Thus the map $\partial_{[1]} \colon \mathcal{C}_0(B_1)\to \mathcal{C}_0(B_0)$ is zero.
For $p\geq 1$, the degeneracy condition implies that
$\sigma_+$ and $\sigma_-$ form chains in $B_{0}\approx S^{1}$ that follow opposite orientations.
Therefore, $\partial_{[1]}(\sigma)$ determines the trivial homology class in $H_p(B_0)$.
In particular, even if $c\in\mathcal{C}_1(B_1)$ is a chain representing a generator of $H_1(B_1)\cong \ZZ$,
$\partial_{[1]}(c)$ still defines the trivial homology class in $H_1(B_0)\cong \ZZ$.
Since the row of $B_0\approx S^1$ is the singular simplicial chain complex up to sign,
there exists a chain $d\in\mathcal{C}_2(B_0)$
such that $\partial_{[0]}(d)=-\partial_{[1]}(c)$.
Hence, $(c,d)\in\mathcal{C}_1(B_1)\oplus\mathcal{C}_2(B_0)$ generates $\HB_2(f,\mathcal{R})\cong\ZZ$.
Moreover, generators of each $H_0(B_1)\cong\ZZ$ and $H_1(B_0)\cong\ZZ$ generate $\HB_1(f,\mathcal{R})\cong \ZZ\oplus \ZZ$.
\end{example}


\begin{example}\label{example:Morse--Bott func 4}
Assume that $M=\mathbb{T}^2$ and that
$f\colon\mathbb{T}^2\to\RR$ is a function whose critical submanifolds and the negative gradient flows,
reducing the Morse--Bott index of $f$ by one, are as shown in Figure \ref{figure:deformed_T^2_crit}.
We note that, unlike the height function, the negative gradient flow lines from $B_1$ to $B_0$
are slightly tilted in each region where the curvature of the torus is positive or negative.

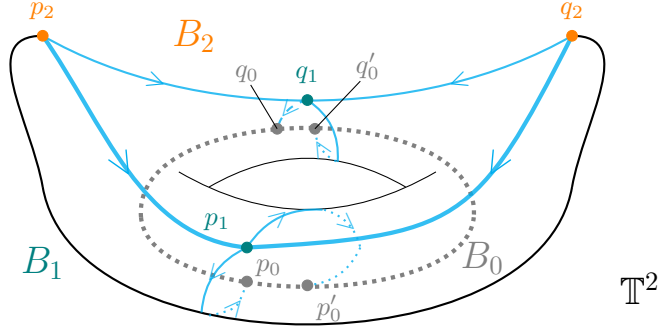
\begin{figure}[ht]
\centering
\begin{tikzpicture}
    \draw[ultra thick,gray,dotted](-2.2,-0.4) to [out=95,in=85] [distance=45pt](2.2,-0.4);
    \draw[ultra thick,gray,dotted](-2.2,-0.4) to [out=-85,in=-95] [distance=35pt](2.2,-0.4);
    \draw[thick](-3.5,0) to [out=-80,in=-100] [distance=70pt](3.5,0);    
    \draw[thick,cyan,opacity=0.8](-3.5,2) to [out=-10,in=-170] [bend right,distance=65pt](3.5,2);
    \draw[ultra thick,cyan,opacity=0.8](-3.5,2) to [out=-55,in=175] [distance=40pt](-0.8,-0.8);
    \draw[ultra thick,cyan,opacity=0.8](3.5,2) to [out=-135,in=0] [bend left ,distance=80pt](-0.8,-0.8);
    \draw[thick](-3.5,0) to [out=100,in=-180] [distance=25pt](-3.5,2);
    \draw[thick](3.5,0) to [out=80,in=0] [distance=25pt](3.5,2);
    \draw (-1.7,0.2) to [out=-30,in=-150] (1.7,0.2);
    \draw (-1.3,0) to [out=30,in=150] (1.3,0);
    \draw[thick,cyan,opacity=0.8](-0.8,-0.8) to [out=-140,in=80] [bend right]node{\rotatebox{50}{$<$}}(-1.4,-1.66);
    \draw[thick,cyan,dotted,opacity=0.8](-1.3,-1.7) to [out=40,in=-120] [bend right,distance=11pt]node{\rotatebox{50}{$>$}}(-0.8,-1.25);
    \draw[thick,cyan,opacity=0.8](-0.8,-0.8) to [out=30,in=-160] [bend left]node{\rotatebox{40}{$>$}}(0.2,-0.3);
    \draw[thick,cyan,dotted,opacity=0.8](0.22,-0.3) to [out=0,in=10] [bend left]node{\rotatebox{-40}{$>$}}(0.7,-0.8);
    \draw[thick,cyan,dotted,opacity=0.8](0.7,-0.8) to [out=-90,in=10] [bend left](0,-1.3);
    \draw[thick,cyan,dashed,opacity=0.8](0,1.15) to [out=-140,in=60] [bend right]node{\rotatebox{50}{$<$}}(-0.4,0.77);
    \draw[thick,cyan,opacity=0.8](0,1.15) to [out=-10,in=80] [bend left](0.4,0.33);
    \draw[thick,cyan,dotted,opacity=0.8](0.38,0.34) to [out=90,in=-80] [bend left]node{\rotatebox{-60}{$<$}}(0.1,0.77);
    \draw (1.9,-0.6) node[below right,gray,thick]{{\LARGE$B_0$}};
    \draw (-3.5,-1) node[teal,thick]{{\LARGE$B_1$}};
    \draw (-1.5,2) node[orange,thick]{{\LARGE$B_2$}};
    \draw (4,-1) node[below right,thick]{{\LARGE$\mathbb{T}^2$}};
    \draw[fill,orange] (-3.5,2) circle (2pt) node [above=1mm] {$p_2$};
    \draw[fill,orange] (3.5,2) circle (2pt) node [above=1mm] {$q_2$};
    \draw[thick,cyan] (-2.0,1.4) node{\rotatebox{-20}{$>$}};
    \draw[thick,cyan] (2.0,1.4) node{\rotatebox{20}{$<$}};
    \draw[ultra thick,cyan] (-2.44,0.3) node{\rotatebox{-60}{\LARGE$>$}};
    \draw[ultra thick,cyan] (2.5,0.3) node{\rotatebox{60}{\LARGE$<$}};
    \draw[fill,teal] (-0.8,-0.8) circle (2pt) node [above left=1mm] {$p_1$};
    \draw[fill,teal] (0,1.15) circle (2pt) node [above=1mm] {$q_1$};
    \draw[fill,gray] (-0.8,-1.25) circle (2pt);
    \draw[gray] (-0.6,-1.1) node [above right=-2mm] {$p_0$};
    \draw[fill,gray] (0,-1.3) circle (2pt) node [below right] {$p'_0$};
    \draw[fill,gray] (0.1,0.77) circle (2pt);
    \draw (0.1,0.77) -- (0.7,1.4);
    \draw[gray] (0.8,1.6) node{{$q'_0$}};
    \draw[fill,gray] (-0.4,0.77) circle (2pt);
    \draw (-0.4,0.77) -- (-0.75,1.35);
    \draw[gray] (-0.8,1.5) node{{$q_0$}};
\end{tikzpicture}
\caption{The critical submanifolds of $f$.}
\label{figure:deformed_T^2_crit}
\end{figure}

Then we have $B_0\approx S^1$, $B_1=\{p_1,q_1\}$, and $B_2=\{p_2,q_2\}$.
The Morse--Bott--Smale chain complex is as follows:
\[
    \xymatrix@C=25pt@R=25pt{
    \cdots & \mathcal{C}_1(B_2) \ar[r]^-{0} \ar[dr]|-{\partial_{[1]}} \ar[ddr]|(.3){\partial_{[2]}}|\hole \ar@{}[d]|{\oplus} & \ZZ p_2 \oplus \ZZ q_2 \ar[r]^-{\partial_{[0]}} \ar[dr]|-{\partial_{[1]}} \ar[ddr]|(.3){\partial_{[2]}}|\hole \ar@{}[d]|{\oplus}  & 0 \ar@{}[d]|{\oplus} & 0 \ar@{}[d]|{\oplus}& \\
    \cdots & \mathcal{C}_2(B_1) \ar[r]^(.6){\cong} \ar[dr]|-{\partial_{[1]}} \ar@{}[d]|{\oplus}& \mathcal{C}_1(B_1) \ar[r]^(.6){0} \ar[dr]|-{\partial_{[1]}} \ar@{}[d]|{\oplus} &\ZZ p_1 \oplus \ZZ q_1 \ar[r]^-{\partial_{[0]}}\ar[dr]|-{\partial_{[1]}} \ar@{}[d]|{\oplus} & 0 \ar@{}[d]|{\oplus} &  \\ 
    \cdots &\mathcal{C}_3(B_0) \ar[r]^-{\partial_{[0]}} \ar@{}[d]|{||} & \mathcal{C}_2(B_0) \ar[r]^-{\partial_{[0]}} \ar@{}[d]|{||} & \mathcal{C}_1(B_0) \ar[r]^-{\partial_{[0]}} \ar@{}[d]|{||} &\mathcal{C}_0(B_0) \ar[r]^-{\partial_{[0]}} \ar@{}[d]|{||} & 0 \\
    \cdots & \CB_3(f,\mathcal{R}) \ar[r]^-{\bm{\partial}_3} & \CB_2(f,\mathcal{R}) \ar[r]^-{\bm{\partial}_2} & \CB_1(f,\mathcal{R}) \ar[r]^-{\bm{\partial}_1} & \CB_0(f,\mathcal{R}) \ar[r]^-{\bm{\partial}_0} & 0
    }
\]
It is straightforward to show that $\HB_{k}(f,\mathcal{R})=0$ for any $k\geq 3$.

By assigning appropriate orientations to the manifold $M$,
the critical submanifolds, and their negative normal bundles,
assume that the moduli spaces $\overline{\mathcal{M}}(B_2,B_0)$ and $\overline{\mathcal{M}}(B_1,B_0)$
are oriented as shown in Figure \ref{figure:deformed_T^2_moduli}.

\begin{figure}[ht]
\centering
\begin{tikzpicture}
    \draw[thick,dotted](-2.2,-0.4) to [out=95,in=85] [distance=45pt](2.2,-0.4);
    \draw[thick,dotted](-2.2,-0.4) to [out=-85,in=-95] [distance=35pt](2.2,-0.4);
    \draw[thick](-3.5,0) to [out=-80,in=-100] [distance=70pt](3.5,0);    
    \draw[thick](-3.5,2) to [out=-10,in=-170] [bend right,distance=65pt](3.5,2);
    \draw[thick](-3.5,2) to [out=-55,in=175] [distance=40pt](-0.8,-0.8);
    \draw[thick](3.5,2) to [out=-135,in=0] [bend left ,distance=80pt](-0.8,-0.8);
    \draw[thick](-3.5,0) to [out=100,in=-180] [distance=25pt](-3.5,2);
    \draw[thick](3.5,0) to [out=80,in=0] [distance=25pt](3.5,2);
    \draw (-1.7,0.2) to [out=-30,in=-150] (1.7,0.2);
    \draw (-1.3,0) to [out=30,in=150] (1.3,0);
    \draw[thick](-0.8,-0.8) to [out=-140,in=80] [bend right]node{\rotatebox{50}{$<$}}(-1.4,-1.66);
    \draw[thick,dotted](-1.3,-1.7) to [out=40,in=-120] [bend right,distance=11pt]node{\rotatebox{50}{$>$}}(-0.8,-1.25);
    \draw[thick](-0.8,-0.8) to [out=30,in=-160] [bend left]node{\rotatebox{40}{$>$}}(0.2,-0.3);
    \draw[thick,dotted](0.22,-0.3) to [out=0,in=10] [bend left]node{\rotatebox{-40}{$>$}}(0.7,-0.8);
    \draw[thick,dotted](0.7,-0.8) to [out=-90,in=10] [bend left](0,-1.3);
    \draw[thick,dashed](0,1.15) to [out=-140,in=60] [bend right]node{\rotatebox{50}{$<$}}(-0.4,0.77);
    \draw[thick](0,1.15) to [out=-10,in=80] [bend left](0.4,0.33);
    \draw[thick,dotted](0.38,0.34) to [out=90,in=-80] [bend left]node{\rotatebox{-60}{$<$}}(0.1,0.77);
    \draw (3,-1) node[below right,thick]{{\LARGE$\mathbb{T}^2$}};
    \draw[fill] (-3.5,2) circle (2pt) node [above=1mm] {$p_2$};
    \draw[fill] (3.5,2) circle (2pt) node [above=1mm] {$q_2$};
    \draw[thick] (-2.0,1.4) node{\rotatebox{-20}{$>$}};
    \draw[thick] (2.0,1.4) node{\rotatebox{20}{$<$}};
    \draw[ultra thick] (-2.44,0.3) node{\rotatebox{-60}{\LARGE$>$}};
    \draw[ultra thick] (2.5,0.3) node{\rotatebox{60}{\LARGE$<$}};
    \draw[fill] (-0.8,-0.8) circle (2pt) node [above left=1mm] {$p_1$};
    \draw[fill] (0,1.15) circle (2pt) node [above=1mm] {$q_1$};
    \draw[fill] (-0.8,-1.25) circle (2pt);
    \draw (-0.6,-1.1) [above right=-2mm] node{$p_0$};
    \draw[fill] (0,-1.3) circle (2pt) node [below right] {$p'_0$};
    \draw[fill] (0.1,0.77) circle (2pt);
    \draw (0.1,0.77) -- (0.7,1.4);
    \draw (0.8,1.6) node{{$q'_0$}};
    \draw[fill] (-0.4,0.77) circle (2pt);
    \draw (-0.4,0.77) -- (-0.75,1.35);
    \draw (-0.8,1.5) node{{$q_0$}};
    \draw[ultra thick,blue](-2.3,0.1) to [out=80,in=-100] node{\LARGE\rotatebox{50}{$>$}}(-1.5,1.3);
    \draw[fill,blue] (-2.3,0.1) circle (2pt);
    \draw[fill,blue] (-1.5,1.3) circle (2pt);
    \draw[ultra thick,blue](-3.65,0.5) to [out=0,in=-170] [bend right]node{\LARGE\rotatebox{0}{$>$}}(-2.55,0.5);
    \draw[ultra thick,blue,dotted](-3.6,0.5) to [out=80,in=-170] (-2.5,1.55);
    \draw[fill,blue](-2.5,1.55) circle (2pt);
    \draw[fill,blue](-2.55,0.5) circle (2pt);
    \draw[ultra thick,blue](1.8,1.35) to [out=0,in=-170] [bend right]node{\LARGE\rotatebox{-50}{$>$}}(2.65,0.5);
    \draw[fill,blue](1.8,1.35)  circle (2pt);
    \draw[fill,blue](2.65,0.5)  circle (2pt);
    \draw[ultra thick,blue](3,1.1) to [out=0,in=-170] [bend right]node{\LARGE\rotatebox{20}{$>$}}(3.85,1.2);
    \draw[fill,blue](3,1.1)  circle (2pt);
    \draw[ultra thick,blue,dotted](3.85,1.2) to [out=150,in=10] [bend right](2.5,1.55);
    \draw[fill,blue](2.5,1.55)  circle (2pt);
    \draw[fill,red](-0.3,1)  circle (2pt);
    \draw[fill,red](0.35,0.7)  circle (2pt);
    \draw[fill,red](0.1,-0.3)  circle (2pt);
    \draw[fill,red](-1.3,-1.3)  circle (2pt);
    \draw (0,2.2) node[blue,thick]{{\LARGE$\overline{\mathcal{M}}(B_{2},B_{0})$}};
    \draw (-2.5,-2) node[red,thick]{{\LARGE$\overline{\mathcal{M}}(B_{1},B_{0})$}};
\end{tikzpicture}
\caption{The moduli spaces $\overline{\mathcal{M}}(B_2,B_0)$ and $\overline{\mathcal{M}}(B_1,B_0)$.}
\label{figure:deformed_T^2_moduli}
\end{figure}
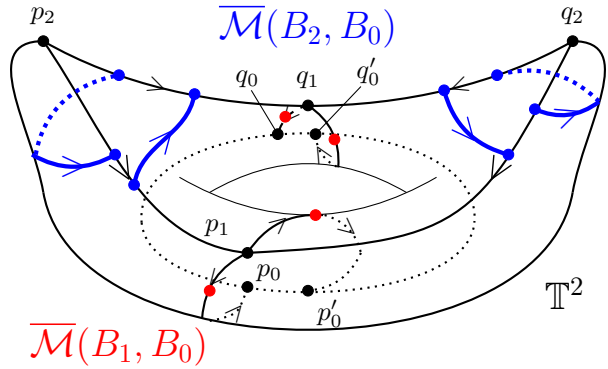

Let $[kp_2+\ell q_2,c,d]\in\HB_2(f,\mathcal{R})$.
We note that $[kp_2+\ell q_2,c,d] = [kp_2+\ell q_2,0,d']$ for some $d'$.
Indeed, for any $c\in\mathcal{C}_1(B_1)$ we have
\[
    \left(0,c,\partial_{[1]}\bigl(\partial_{[0]}^{-1}(c)\bigr)\right) = \bm{\partial}_3\left(0,\partial_{[0]}^{-1}(c),0\right) \in \Image{\bm{\partial}_3}\subset (\ZZ p_2 \oplus \ZZ q_2)\oplus \mathcal{C}_1(B_1)\oplus \mathcal{C}_2(B_0)
\]
Since $\bm{\partial}_2(k p_2+\ell q_2,0,d') = 0$,
we have $\partial_{[2]}(k p_2+\ell q_2)+\partial_{[0]}(d')=0$
and
\[
    0=\partial_{[1]}(k p_2+\ell q_2)=\pm(k-\ell )(p_1 -q_1),
\]
where the sign is determined by the orientation of the negative normal bundle of $B_{1}$.
Thus, we obtain $k=\ell$.
Therefore, every generator of $\HB_2(f,\mathcal{R})$ is of the form either
$[(p_2+q_2,0,*)]$ or $[(0,0,*)]$.
However, since $H_2(B_1;\ZZ)=0$, the latter cannot generate $\HB_2(f,\mathcal{R})$.

Conversely, we claim that there exists a singular simplicial chain $e\in \mathcal{C}_2(B_0)$
such that $\partial_{[0]}(e)=\partial_{[2]}(p_2+q_2)$. 
First we compute $\partial_{[2]}(p_2)=\sigma_+ + \sigma_-$ as in Example \ref{example:Perfect Morse--Bott func}.
Since the orientations of the connected component of the moduli space
$\overline{\mathcal{M}}(B_2,B_0)$ are reversed on the domains of $\sigma_+$ and $\sigma_-$,
these domains inherit opposite orientations.
Hence, via the corresponding representing chains,
$\sigma_+$ and $\sigma_-$ become chains in $B_{0}\approx S^1$ that follow opposite orientations.
Similarly, we have $\partial_{[2]}(q_2)=\tau_+ +\tau_-$ for suitable chains $\tau_+$ and $\tau_-$ in $B_0$.
In this case, both $\sigma_+ +\tau_+$ and $\sigma_- +\tau_-$ represent generators of $H_1(B_0)\cong \ZZ$,
whereas $\partial_{[2]}(p_2+q_2)$ represents the trivial homology class in $H_1(B_0)$.
Since the Morse--Bott--Smale chain complex restricted to the row of $B_{0}$ coincides with the singular simplicial chain complex of $B_0$ up to sign,
there exists a chain $e\in \mathcal{C}_2(B_0)$ such that $\partial_{[0]}(e)=\partial_{[2]}(p_2+q_2)$.
Hence $(p_2+q_2,0,-e)\in\Ker{\bm{\partial}_2}$.
On the other hand,
since $\partial_{[0]}\colon\mathcal{C}_1(B_2)\to\ZZ p_2\oplus\ZZ q_2$ is zero,
we have $(p_2+q_2,0,-e)\not\in\Image{\bm{\partial}_3}$.
Therefore, $\HB_2(f,\mathcal{R})=\ZZ[(p_2+q_2,0,-e)]\cong\ZZ$.

We note that, depending on the chosen orientation,
$\partial_{[1]}(p_1)=p_0-p'_0$ and $\partial_{[1]}(q_1)=q_0-q'_0$ hold.
Then one can confirm that $\HB_0(f,\mathcal{R})\cong\ZZ$.
Now we want to find a non-trivial homology class in $\HB_1(f,\mathcal{R})$.
Let $(k p_1+\ell q_1,c) \in \Ker{\bm{\partial}_1} \subset \mathcal{C}_0(B_1)\oplus \mathcal{C}_1(B_0)$.
Then $\partial_{[1]}(kp_1 +\ell q_1)+\partial_{[0]}(c)=0$.
Since $\partial_{[1]} \colon \ZZ p_1 \oplus \ZZ q_1 \to \mathcal{C}_0(B_0)$ is injective,
if $c$ is a cycle of a generator of $H_1(B_0;\ZZ)\cong\ZZ$,
then we have $k=\ell =0$.
In this case, the class $\alpha_3 =[(0,c)]$ defines a non-trivial homology class in $\HB_1(f,\mathcal{R})$.
On the other hand,
since we know that
\[
    \Image{\bigl(\partial_{[1]} \colon \ZZ p_2 \oplus \ZZ q_2 \to \ZZ p_1 \oplus \ZZ q_1\bigr)} = \ZZ \left(p_1-q_1\right),
\]
it is enough to consider the case of $\ell=0$.
The condition $\bm{\partial}_1(kp_1,c) =0$ implies that $c$ is represented by a chain $k\sigma$,
where $\sigma\colon \Delta ^1 \to B_0$ is a smooth map satisfying $\sigma(0)=p'_0$ and $\sigma(1)=p_0$.
We note that the singular 1-simplex $\sigma$ in $B_0$ is not a cycle.
Let $\sigma_1$, $\sigma_2 \colon \Delta^1 \to B_0$ be maps from $p'_0$ to $p_0$
such that the difference $\sigma_1 -\sigma_2$ represents a generator of $H_1(B_0;\ZZ)\cong\ZZ$.
Then the classes $\alpha_1=[(p_1,\sigma_1)]$ and $\alpha_2=[(p_1,\sigma_2)]$
define non-trivial homology classes in $\HB_1(f,\mathcal{R})$.
Since $\alpha_1-\alpha_2=\alpha_3$, it follows that $\alpha_1 \neq \alpha_2$.
Therefore, we conclude that
\[
    \HB_1(f,\mathcal{R})
    = \frac{\ZZ\alpha_1 \oplus \ZZ\alpha_2 \oplus \ZZ\alpha_3}{\ZZ(\alpha_1 -\alpha_2 - \alpha_3)}%
    \cong \ZZ\oplus\ZZ.
\]
\end{example}


\section{Application: Another proof of the Morse Homology Theorem}\label{section:A proof of the MH thm}

In this section, we provide another proof of the Morse Homology Theorem.
Let $(M,g)$ be a closed oriented Riemannian manifold.
Let $f\colon M\to\RR$ be a Morse--Smale function on $(M,g)$.
Let $B_i$ be the set of critical points of $f$ on index $i$.
We assume that the unstable manifolds of all critical points of $f$ are oriented.
We fix a good representing chain system $\mathcal{R}$ for $f$, which consists of relative cycles
that represent the positive relative fundamental class.
In particular, when a set $P$ has dimension zero,
the constant map with the same sign as the orientation on $P$ is chosen as its representing chain.
Consider the Morse--Bott--Smale chain complex $(\CB_*(f,\mathcal{R}),\bm{\partial}_*)$ for $f$.

\begin{lemma}\label{lemma:diagonal_is_MSW}
For any $k=0$,~$1$, \dots, $\dim{M}$
the homomorphism $\partial_{[1]}\colon\mathcal{C}_0(B_k)\to\mathcal{C}_0(B_{k-1})$
can be identified with the homomorphism
$\partial_k^{\mathrm{Morse}}\colon\CM_k(f)\to\CM_{k-1}(f)$
$($recall Section \ref{Section:Morse homology} for the definition$)$.
Moreover, the pair $(\mathcal{C}_0(B_*),\{\partial_{[1]}\})$ is a chain complex.
In particular, the pair $(\CM_*(f),\partial_*^{\mathrm{Morse}})$
is a chain complex.
\end{lemma}

\begin{proof}
It is obvious that $\mathcal{C}_0(B_k)=\CM_k(f)$
by identifying a constant map $\sigma\colon\Delta^0\to B_k$ with its value $q=\sigma(\Delta^0)$.
Here we note that other maps $\dotDelta^0\to B_k$ taking the same value $q$
are equivalent to each other due to the degeneracy condition (Definition \ref{definition:deg con}).
We recall that the homomorphism $\partial_{[1]}\colon\mathcal{C}_0(B_k)\to\mathcal{C}_0(B_{k-1})$
sends a constant map $\sigma\colon\Delta^0\to B_k$ to the map
$\mathrm{ev}_+\circ\mathrm{pr}_2\colon\Delta^0\times_{B_k}\mathcal{M}(B_k,B_{k-1})\to B_{k-1}$,
which is the formal sum
\[
    \sum_{p\in B_{k-1}}\bigl(\Delta^0\times_{\{q\}}\mathcal{M}(\{q\},\{p\})\to \{p\}\bigr),
\]
where $q=\sigma(\Delta^0)$.
Therefore, we deduce that
$\partial_{[1]}(\sigma)=\partial_k^{\mathrm{Morse}}(q)$.

Now we shall show that the square
$\partial_{[1]}\circ\partial_{[1]}\colon\mathcal{C}_0(B_k)\to\mathcal{C}_0(B_{k-2})$ is zero.
We recall that the restriction of the Morse--Bott--Smale chain complex for $f$ to each row
\[
    \cdots \xrightarrow{\partial_{[0]}} \mathcal{C}_{k+1}(B_i) \xrightarrow{\partial_{[0]}} \mathcal{C}_k(B_i) \xrightarrow{\partial_{[0]}} \mathcal{C}_{k-1}(B_i) \xrightarrow{\partial_{[0]}} \cdots
\]
coincides with the (smooth) singular simplicial chain complex of $B_i$ up to sign.
Since $f$ is Morse, the set $B_i$ is finite,
and hence the singular homology group $H_k(B_i;\ZZ)$ vanishes whenever $k>0$.
Therefore, the homomorphism $\partial_{[0]}\colon\mathcal{C}_k(B_i)\to\mathcal{C}_{k-1}(B_i)$
is an isomorphism if $k$ is an even positive integer, and it is zero otherwise.
Thus, the Morse--Bott--Smale chain complex for the Morse--Smale function $f$ is as follows. 
\[
    \xymatrix@C=20pt@R=20pt{
    \ddots & \vdots \ar@{}[d]|{\oplus} &\vdots \ar@{}[d]|{\oplus} &\vdots \ar@{}[d]|{\oplus} &\vdots \ar@{}[d]|{\oplus} & \\
    \cdots & \mathcal{C}_1(B_2) \ar[r]^-{0} \ar[dr]|-{\partial_{[1]}} \ar[ddr]|(.3){\partial_{[2]}}|\hole \ar@{}[d]|{\oplus} & \mathcal{C}_0(B_2) \ar[r]^-{\partial_{[0]}} \ar[dr]|-{\partial_{[1]}} \ar[ddr]|(.3){\partial_{[2]}}|\hole \ar@{}[d]|{\oplus} & 0 \ar@{}[d]|{\oplus} & 0\ar@{}[d]|{\oplus} & \\
    \cdots & \mathcal{C}_2(B_1) \ar[r]^(.6){\cong} \ar[dr]|-{\partial_{[1]}} \ar@{}[d]|{\oplus}& \mathcal{C}_1(B_1) \ar[r]^(,6){0} \ar[dr]|-{\partial_{[1]}} \ar@{}[d]|{\oplus} & \mathcal{C}_0(B_1) \ar[r]^-{\partial_{[0]}} \ar[dr]|-{\partial_{[1]}} \ar@{}[d]|{\oplus} & 0 \ar@{}[d]|{\oplus} & \\ 
    \cdots & \mathcal{C}_3(B_0) \ar[r]^-{0} \ar@{}[d]|{||} & \mathcal{C}_2(B_0) \ar[r]^-{\cong} \ar@{}[d]|{||} & \mathcal{C}_1(B_0) \ar[r]^-{0} \ar@{}[d]|{||} & \mathcal{C}_0(B_0) \ar[r]^-{\partial_{[0]}} \ar@{}[d]|{||} & 0 \\
    \cdots & \CB_3(f,\mathcal{R}) \ar[r]^-{\bm{\partial}_3} & \CB_2(f,\mathcal{R}) \ar[r]^-{\bm{\partial}_2} & \CB_1(f,\mathcal{R}) \ar[r]^-{\bm{\partial}_1} & \CB_0(f,\mathcal{R}) \ar[r]^-{\bm{\partial}_0} & 0
    }
\]
Therefore, Proposition \ref{proposition:boundary operator} yields that
\[
    0%
    = \partial_{[0]}\circ\partial_{[2]} + \partial_{[1]}\circ\partial_{[1]} + \partial_{[2]}\circ\partial_{[0]}
    = 0\circ\partial_{[2]} + \partial_{[1]}\circ\partial_{[1]} + \partial_{[2]}\circ 0
    = \partial_{[1]}\circ\partial_{[1]},
\]
which completes the proof.
\end{proof}

On the other hand,
we note that the injection $\iota\colon\mathcal{C}_0(B_k)\to\CB_k(f,\mathcal{R})$ given by
\[
    \iota(c_0) = (c_0,0,\ldots,0)%
    \in \mathcal{C}_0(B_k) \oplus \mathcal{C}_1(B_{k-1}) \oplus \cdots \oplus \mathcal{C}_k(B_0) \cong \CB_k(f,\mathcal{R}),
\]
for $c_0\in\mathcal{C}_0(B_k)$,
does \textit{not} induce a homomorphism between the homology groups.
Indeed, for some $c_0\in\Ker{\partial_{[1]}}$,
its image $\iota(c_0)$ does not lie in $\Ker{\bm{\partial}_k}$.

The following is the key lemma which claims
that there exists a canonical embedding of the Morse--Smale--Witten chain complex into
the Morse--Bott--Smale chain complex.

\begin{lemma}\label{lemma:canonical embedding}
Let $k=0$,~$1$, \dots,~$\dim{M}$.
For any singular simplicial chain $c_0\in\mathcal{C}_0(B_k)$
there exists a unique
$(c_1,\ldots,c_k)\in\mathcal{C}_1(B_{k-1})\oplus\cdots\oplus\mathcal{C}_k(B_0)$
that satisfies the following conditions:

\begin{enumerate}
    \item If $i$ is odd, then $c_i=0$.
    \item Assume that $(d_0,d_1,\ldots,d_{k-1})=\bm{\partial}_k(c_0,c_1,\ldots,c_k)$ as elements of
    \[
        \mathcal{C}_0(B_{k-1}) \oplus \mathcal{C}_1(B_{k-2}) \oplus \cdots \oplus \mathcal{C}_{k-1}(B_0).
    \]
    If $i$ is odd, then $d_i=0$.
    If $i$ is even, then
    \[
        d_i = -\partial_{[0]}^{-1}%
        \left(%
        \partial_{[i]}(d_0) + \partial_{[i-2]}(d_2) + \partial_{[i-4]}(d_4) + \cdots + \partial_{[2]}(d_{i-2})%
        \right).
    \]
\end{enumerate}
\end{lemma}

\begin{proof}
Let $c_0\in\mathcal{C}_0(B_k)$.
First we show the existence of $c_1$, \dots, $c_k$ satisfying the desired conditions.
When $i$ is either negative or odd, we set $c_i=0$.
When $i$ is positive and even, we define
\[
    c_i = -\partial_{[0]}^{-1}%
    \left(%
    \partial_{[i]}(c_0) + \partial_{[i-2]}(c_2) + \partial_{[i-4]}(c_4) + \cdots + \partial_{[2]}(c_{i-2})%
    \right)
\]
inductively.
Here we note that $\partial_{[0]}\colon\mathcal{C}_i(B_{k-i})\to\mathcal{C}_{i-1}(B_{k-i})$ is an isomorphism
whenever $i$ is positive and even.
If $i$ is odd, then we have
\begin{align*}
    d_i%
    &= \partial_{[i+1]}(c_0) + \partial_{[i]}(c_1) + \partial_{[i-1]}(c_2) + \cdots + \partial_{[1]}(c_i) + \partial_{[0]}(c_{i+1}) \\
    &= \partial_{[i+1]}(c_0) + 0 + \partial_{[i-1]}(c_2) + \cdots + 0 + \partial_{[0]}(c_{i+1}) \\
    &= 0.
\end{align*}
If $i$ is positive and even, then Proposition \ref{proposition:boundary operator} yields that
\begin{align*}
    \partial_{[0]}(d_i)%
    &= \partial_{[0]} \left(\sum_{j=0}^{i+1} \partial_{[j]}(c_{i+1-j})\right)%
    = \sum_{j=1}^{i+1} \bigl(\partial_{[0]} \circ \partial_{[j]}\bigr)(c_{i+1-j}) \\
    &= \sum_{j=1}^{i+1} \left(%
    -\sum_{q=1}^j \bigl(\partial_{[q]} \circ \partial_{[j-q]}\bigr)(c_{i+1-j})%
    \right) \\
    &= \sum_{q=1}^{i+1} \left(%
    -\sum_{j=q}^{i+1} \bigl(\partial_{[q]} \circ \partial_{[j-q]}\bigr)(c_{i+1-j})%
    \right) \\
    &= -\sum_{q=1}^{i+1} \partial_{[q]} \left(%
    \sum_{j=q}^{i+1} \partial_{[j-q]}(c_{i+1-j})%
    \right)%
    = -\sum_{q=1}^i \partial_{[q]}%
    (d_{i-q}) \\
    &= -\left(\partial_{[i]}(d_0) + \partial_{[i-1]}(d_1) + \partial_{[i-2]}(d_2) + \cdots + \partial_{[2]}(d_{i-2}) + \partial_{[1]}(d_{i-1})\right) \\
    &= -\left(\partial_{[i]}(d_0) + 0 + \partial_{[i-2]}(d_2) + \cdots + \partial_{[2]}(d_{i-2}) + 0\right) \\
    &= -\left(\partial_{[i]}(d_0) + \partial_{[i-2]}(d_2) + \cdots + \partial_{[2]}(d_{i-2})\right),
\end{align*}
where the fourth equation is obtained by changing the order of summation as illustrated in Figure \ref{figure:order_of_sum}.
Therefore, we obtain
\[
    d_i = -\partial_{[0]}^{-1}%
    \left(%
    \partial_{[i]}(d_0) + \partial_{[i-2]}(d_2) + \partial_{[i-4]}(d_4) + \cdots + \partial_{[2]}(d_{i-2})%
    \right).
\]
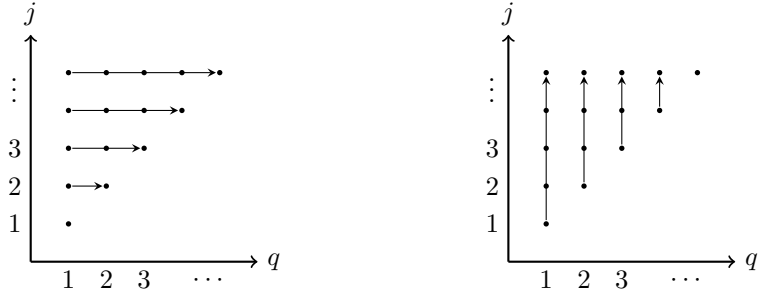
\begin{figure}[htbp]
\centering
\begin{minipage}[b]{.49\columnwidth}
    \centering
    \begin{tikzpicture}[scale=.5]
    \draw[thick, ->] (0,0) -- (6,0) node [right] {$q$};
    \draw[thick, ->] (0,0) -- (0,6) node [above] {$j$};
    \foreach \q in {1,2,3}
       \node at (\q,0) [below] {$\q$};
    \foreach \j in {1,2,3}
       \node at (0,\j) [left] {$\j$};
    \foreach \q in {1,2,3,4,5}
        \fill (\q,5) circle (2pt) ;
    \foreach \q in {1,2,3,4}
        \fill (\q,4) circle (2pt) ;
    \foreach \q in {1,2,3}
        \fill (\q,3) circle (2pt) ;
    \foreach \q in {1,2}
        \fill (\q,2) circle (2pt) ;
    \foreach \q in {1}
        \fill (\q,1) circle (2pt) ;
    \node at (4.75,0) [below=1.25pt] {$\cdots$};
    \node at (0,4.75) [left=1.25pt] {$\vdots$};
    \draw[->,>=stealth] (1.1,2) -- (1.9,2);
    \draw[->,>=stealth] (1.1,3) -- (2.9,3);
    \draw[->,>=stealth] (1.1,4) -- (3.9,4);
    \draw[->,>=stealth] (1.1,5) -- (4.9,5);
    \end{tikzpicture}
\end{minipage}
\begin{minipage}[b]{.49\columnwidth}
    \centering
    \begin{tikzpicture}[scale=.5]
    \draw[thick, ->] (0,0) -- (6,0) node [right] {$q$};
    \draw[thick, ->] (0,0) -- (0,6) node [above] {$j$};
    \foreach \q in {1,2,3}
       \node at (\q,0) [below] {$\q$};
    \foreach \j in {1,2,3}
       \node at (0,\j) [left] {$\j$};
    \foreach \q in {1,2,3,4,5}
        \fill (\q,5) circle (2pt) ;
    \foreach \q in {1,2,3,4}
        \fill (\q,4) circle (2pt) ;
    \foreach \q in {1,2,3}
        \fill (\q,3) circle (2pt) ;
    \foreach \q in {1,2}
        \fill (\q,2) circle (2pt) ;
    \foreach \q in {1}
        \fill (\q,1) circle (2pt) ;
    \node at (4.75,0) [below=1.25pt] {$\cdots$};
    \node at (0,4.75) [left=1.25pt] {$\vdots$};
    \draw[->,>=stealth] (1,1.1) -- (1,4.9);
    \draw[->,>=stealth] (2,2.1) -- (2,4.9);
    \draw[->,>=stealth] (3,3.1) -- (3,4.9);
    \draw[->,>=stealth] (4,4.1) -- (4,4.9);
    \end{tikzpicture}
\end{minipage}
\caption{Rearranging the order of summation.}
\label{figure:order_of_sum}
\end{figure}
Moreover, we have $d_0=\partial_{[0]}(c_1)+\partial_{[1]}(c_0)=\partial_{[1]}(c_0)$.

For the uniqueness, we choose another $(c'_1,\ldots,c'_k)$ satisfying the desired conditions.
Let $(d'_0,d'_1,\ldots,d'_{k-1})=\bm{\partial}_k(c_0,c'_1,\ldots,c'_k)$.
If $i$ is odd, then clearly $c'_i=0=c_i$ holds.
If $i$ is positive and even, then
\begin{align*}
    0= d'_{i-1}%
    &= \partial_{[0]}(c'_i) + \partial_{[1]}(c'_{i-1}) +\partial_{[2]}(c'_{i-2}) + \cdots + \partial_{[i-1]}(c'_1) + \partial_{[i]}(c_0) \\
    &= \partial_{[0]}(c'_i) + 0 + \partial_{[2]}(c'_{i-2}) + \cdots + 0 + \partial_{[i]}(c_0),
\end{align*}
which implies that
\[
    c'_i = -\partial_{[0]}^{-1}%
    \left(\partial_{[2]}(c'_{i-2}) + \partial_{[4]}(c'_{i-4}) + \cdots + \partial_{[i]}(c_0)\right).
\]
Therefore, $c'_i=c_i$ holds for all $i$ inductively.
This completes the proof.
\end{proof}

By Lemma \ref{lemma:canonical embedding},
for each $k=0$,~1, \dots,~$\dim{M}$
one can define a map
$\Phi_{\#}\colon\mathcal{C}_0(B_k)\to\CB_k(f,\mathcal{R})$
by the formula
\[
    \Phi_{\#}(c_0) = (c_0,c_1,\ldots,c_k)%
    \in \mathcal{C}_0(B_k) \oplus \mathcal{C}_1(B_{k-1}) \oplus \cdots \oplus \mathcal{C}_k(B_0) \cong \CB_k(f,\mathcal{R}),
\]
for $c_0\in\mathcal{C}_0(B_k)$.

\begin{theorem}\label{theorem:Morse_MB_isom}
Let $f\colon M\to\RR$ be a Morse--Smale function on $(M,g)$.
For any $k=0$,~$1$, \dots,~$\dim{M}$
the map $\Phi_{\#}\colon\mathcal{C}_0(B_k)\to\CB_k(f,\mathcal{R})$ is a chain map.
Moreover, $\Phi_{\#}$ induces an isomorphism
$\Phi_* \colon \HM_k(f) \xrightarrow{\cong} \HB_k(f,\mathcal{R})$.
\end{theorem}

\begin{proof}
We shall show that $\Phi_{\#}\circ\partial_1=\bm{\partial}_k\circ\Phi_{\#}$.
Let $c_0\in\mathcal{C}_0(B_{k+1})$.
Then we can write
\[
    \Phi_{\#}\bigl(\partial_{[1]}(c_0)\bigr)=\bigl(\partial_{[1]}(c_0),c_1,\ldots,c_k\bigr),
\]
and
\[
    \bm{\partial}_k\bigl(\Phi_{\#}(c_0)\bigr)=\bm{\partial}_k(c_0,*,\ldots,*)=\bigl(\partial_{[1]}(c_0),d_1,\ldots,d_k\bigr).
\]
If $i$ is odd, then $c_i=0=d_i$.
If $i$ is positive and even,
then the definition of $\Phi_{\#}$ and the property of $d_i$ imply that
\[
    c_i = -\partial_{[0]}^{-1}%
    \left(%
    \partial_{[i]}\bigl(\partial_{[1]}(c_0)\bigr) + \partial_{[i-2]}(c_2) + \partial_{[i-4]}(c_4) + \cdots + \partial_{[2]}(c_{i-2})%
    \right),
\]
and
\[
    d_i = -\partial_{[0]}^{-1}%
    \left(%
    \partial_{[i]}\bigl(\partial_{[1]}(c_0)\bigr) + \partial_{[i-2]}(d_2) + \partial_{[i-4]}(d_4) + \cdots + \partial_{[2]}(d_{i-2})%
    \right).
\]
Therefore, $c_i=d_i$ holds for all $i$ inductively.
It shows that $\Phi_{\#}\circ\partial_1=\bm{\partial}_k\circ\Phi_{\#}$.
Thus, the map $\Phi_{\#}$ induces a homomorphism $\Phi_*\colon\HM_k(f)\to\HB_k(f,\mathcal{R})$.

It is obvious that $\Phi_*$ is injective.
Indeed, if a homology class $[c_0]\in\HM_k(f)$ satisfies $\Phi_{\#}(c_0)=\bm{\partial}_k(e_0,*,\ldots,*)$ for some $e_0$,
then $c_0=\partial_{[1]}(e_0)\in\Image{\partial_{k+1}^{\mathrm{Morse}}}$.

Now we show the surjectivity.
Choose $[(c_0,c_1,\ldots,c_k)]\in\HB_k(f,\mathcal{R})$ arbitrarily.
Without loss of generality, we may assume that $c_i=0$ when $i$ is odd.
Indeed, one can find $(e_0,e_1,\ldots,e_{k+1})\in\CB_{k+1}(f,\mathcal{R})$
such that $\bm{\partial}_{k+1}(e_0,e_1,\ldots,e_{k+1})=(0,c_1,0,c_3,\ldots,c_k)$ as follows.
When $i$ is either non-positive or odd, we set $e_i=0$.
When $i$ is positive and even, we define $e_i$ by the formula
\[
    c_{i-1} = \partial_{[i]}(e_0) + \partial_{[i-2]}(e_2) + \cdots + \partial_{[2]}(e_{i-2}) + \partial_{[0]}(e_i)
\]
inductively.
Since $(c_0,c_1,\ldots,c_k)\in\Ker{\bm{\partial}_k}$,
we have $c_0\in\Ker{\partial_{[1]}}$ and
\[
    c_i = -\partial_{[0]}^{-1}%
    \left(%
    \partial_{[i]}\bigl(\partial_{[1]}(c_0)\bigr) + \partial_{[i-2]}(c_2) + \partial_{[i-4]}(c_4) + \cdots + \partial_{[2]}(c_{i-2})%
    \right)
\]
whenever $i$ is positive and even.
Now the definition of $\Phi_{\#}$ yields that
$\Phi_*([c_0])=[\Phi_{\#}(c_0)]=[(c_0,c_1,\ldots,c_k)]$.
This completes the proof.
\end{proof}

\begin{proof}[Proof of Theorem \ref{theorem:Morse homology Theorem}]
Let $f\colon M\to\RR$ be a Morse--Smale function on $(M,g)$
and $\mathcal{R}$ a good representing chain system for $f$.
By Lemma \ref{lemma:diagonal_is_MSW},
the pair $(\CM_*(f),\partial_*^{\mathrm{Morse}})$ is a chain complex.
Now Theorem \ref{theorem:Morse_MB_isom} and Corollary \ref{corollary:MB_is_sing} conclude that
\[
    \HM_*(f) \cong \HB_*(f,\mathcal{R}) \cong H_*(M;\ZZ).\qedhere
\]
\end{proof}


\section{Proof of Theorem \ref{theorem:MB homology theorem}}\label{section:independence}

In this section, we prove Theorem \ref{theorem:MB homology theorem}.
Since the proof proceeds in the same manner as in \cite[Section 6]{BH10},
we focus here specifically on the definition and existence of good representing chain systems for chain maps and chain homotopies,
and present the necessary modifications.

Let $(M,g)$ be a closed oriented Riemannian manifold.


\subsection{Chain map}\label{section:chain map}

Let $f_1$, $f_2\colon M\to\RR$ be Morse--Bott--Smale functions on $(M,g)$.
For each $\heartsuit\in\{1,2\}$ and $i=0$,~1, \dots,~$\dim{M}$
let $B_i^{f_{\heartsuit}}$ denote the critical submanifold of $f_{\heartsuit}$ of Morse--Bott index $i$.
For simplicity, we assume that the components of $B_i^{f_{\heartsuit}}$ have the same dimension.
Let $b_i^{f_{\heartsuit}}=\dim{B_i^{f_{\heartsuit}}}$.

Let $F_{21}\colon M\times \RR\to\RR$ be a smooth function which is strictly decreasing in the second component and satisfies
\[
    F_{21}(\cdot,t) = \begin{cases}%
        f_1 & \text{for $t\ll -1$}, \\
        f_2 - \text{(constant)} & \text{for $t\gg 1$}.
    \end{cases}
\]
Let $\{\varphi_{\alpha}\}_{\alpha\in\RR}$ denote the flow of $-\grad{F_{21}}$ with respect to the product metric on $M\times \RR$.
Let $i$, $j=0$,~1, \dots,~$\dim{M}$.
We set
\[
     W^u_{F_{21}}(B_i^{f_1})%
     = \left\{\,(x,t)\in M\times\RR \relmiddle| \lim_{\alpha\to -\infty} \mathrm{pr}_1\bigl(\varphi_{\alpha}(x,t)\bigr)\in B_i^{f_1}\,\right\}
\]
and
\[
     W^s_{F_{21}}(B_j^{f_2})%
     = \left\{\,(x,t)\in M\times\RR \relmiddle| \lim_{\alpha\to\infty} \mathrm{pr}_1\bigl(\varphi_{\alpha}(x,t)\bigr)\in B_j^{f_2}\,\right\}.
\]
For a generic choice of $F_{21}$ the moduli space
\[
    \mathcal{M}_{F_{21}}(B_i^{f_1},B_j^{f_2}) = \bigl(W^u_{F_{21}}(B_i^{f_1}) \cap W^s_{F_{21}}(B_j^{f_2})\bigr) / \RR
\]
is either empty or a compact oriented smooth manifold with corners of dimension $b_i^{f_1}+i-j$.
As with Theorems \ref{theorem:Gluing} and \ref{theorem:Compactification},
we have gluing and compactification results for $F_{21}$ (cf.\ \cite[Lemmas 6.3 and 6.4]{BH10}).
Let $\overline{\mathcal{M}}_{F_{21}}(B_i^{f_1},B_j^{f_2})$ denote the compactification of $\mathcal{M}_{F_{21}}(B_i^{f_1},B_j^{f_2})$.
We define the degree of $\overline{\mathcal{M}}_{F_{21}}(B_i^{f_1},B_j^{f_2})$ to be $b_i^{f_1}+i-j$.
Moreover, taking into account both the orientation induced on the boundary and the orientation defined on the fibered product,
we define
\begin{align*}
    \partial\left(\overline{\mathcal{M}}_{F_{21}}(B_i^{f_1},B_j^{f_2})\right)%
    =(-1)^{i+b_i^{f_1}}%
    \Biggl(%
    & \sum_{n<i}%
    \overline{\mathcal{M}}(B_i^{f_1},B_n^{f_1}) \times_{B_n^{f_1}} \overline{\mathcal{M}}_{F_{21}}(B_n^{f_1},B_j^{f_2}) \\
    &-\sum_{n>j}%
    \overline{\mathcal{M}}_{F_{21}}(B_i^{f_1},B_n^{f_2}) \times_{B_n^{f_2}} \overline{\mathcal{M}}(B_n^{f_2},B_j^{f_2})%
    \Biggr) ,
\end{align*}
where the fibered products are taken over the beginning and endpoint maps.
Then the map $\partial$ satisfies $\deg{\partial}=-1$ and $\partial\circ\partial=0$.

Let $N>\dim{M}$.
For $p\geq 0$
let $C^{F_{21}}_p$ be the set consisting of $p$-faces of $\Delta^N$
and $p$-dimensional connected components of fibered products of the form
\begin{align*}
    \dotDelta^q%
    \times_{B_{i_1}^{f_1}} \overline{\mathcal{M}}(B_{i_1}^{f_1},B_{i_2}^{f_1})%
    \times_{B_{i_2}^{f_1}} \cdots%
    \times_{B_{i_{n-1}}^{f_1}} \overline{\mathcal{M}}(B_{i_{n-1}}^{f_1},B_{i_n}^{f_1})%
    \times_{B_{i_n}^{f_1}} \overline{\mathcal{M}}_{F_{21}}(B_{i_n}^{f_1},B_{j_1}^{f_2}) \\
    \times_{B_{j_1}^{f_2}} \overline{\mathcal{M}}(B_{j_1}^{f_2},B_{j_2}^{f_2})%
    \times_{B_{j_2}^{f_2}} \cdots%
    \times_{B_{j_{k-1}}^{f_2}} \overline{\mathcal{M}}(B_{j_{k-1}}^{f_2},B_{j_k}^{f_2}),
\end{align*}
where $\dotDelta^q$ is a $q$-face of $\Delta^N$ for some $q\leq p$,
$\dim{M} \geq i_1 > i_2 > \cdots > i_n \geq 0$,
and $\dim{M} \geq j_1 > j_2 > \cdots > j_k \geq 0$.
Let $S^{F_{21}}_p=\ZZ[C^{F_{21}}_p]$.
For $p<0$ or $C^{F_{21}}_p=\emptyset$ we set $S^{F_{21}}_p=0$.
Then the pair $(S^{F_{21}}_*,\partial_*)$ is a chain complex of abstract topological chains.

A representing chain system for $F_{21}$ is defined in the same way as in Definition \ref{definition:rep.sys.}.
As the construction of Morse--Bott--Smale chain complexes,
we specify \textit{good} representing chains for every element of $C^{F_{21}}_p$.
We fix good representing chain systems $\mathcal{R}^{f_1}$ and $\mathcal{R}^{f_2}$ for $f_1$ and $f_2$, respectively.

\begin{definition}\label{definition:good family of rep. chain for F}
A representing chain system $\mathcal{R}^{F_{21}}$ for $F_{21}$ is called \textit{good}
if it satisfies the following conditions.

\begin{enumerate}
    \item For any $P\in C_p^{F_{21}}$
    there exists a unique element $(P,s_P)\in\mathcal{R}^{F_{21}}$
    such that $s_P$ represents the positive relative fundamental class of $P$ in $H_p(P,\partial P)$.
    
    \item Let $Q\in C_q^{f_1}$ be a connected component of a fibered product of the form 
    \[
       \dotDelta%
       \times_{B_{i_1}^{f_1}} \overline{\mathcal{M}}(B_{i_1}^{f_1},B_{i_2}^{f_1})%
       \times_{B_{i_2}^{f_1}} \cdots%
       \times_{B_{i_{n-1}}^{f_1}} \overline{\mathcal{M}}(B_{i_{n-1}}^{f_1},B_{i_n}^{f_1}),
    \]
    where $\dotDelta$ is a face of $\Delta^N$.
    Let $s^{f_1}_Q = \sum_{\alpha} n_{\alpha} \sigma_{\alpha} \in \mathcal{R}^{f_1}$
    be the representing chain for $Q$.
    For each $\alpha$ and any critical submanifold $B^{f_2}$ of $f_2$
    consider the fibered product $\Delta^q_{\alpha} \times_{B_{i_n}^{f_1}} \overline{\mathcal{M}}_{F_{21}}(B_{i_n}^{f_1},B^{f_2})$
    of the composition map
    \[
        \mathrm{ev}_+ \circ \mathrm{pr} \circ \sigma_{\alpha} \colon%
        \Delta^q_{\alpha} \xrightarrow{\sigma_{\alpha}} Q \xrightarrow{\mathrm{pr}} \overline{\mathcal{M}}(B_{i_{n-1}}^{f_1},B_{i_n}^{f_1})%
        \xrightarrow{\mathrm{ev}_+} B_{i_n}^{f_1},
    \]
    and the beginning map $\mathrm{ev}_-\colon\overline{\mathcal{M}}_{F_{21}}(B_{i_n}^{f_1},B^{f_2})\to B_{i_n}^{f_1}$.
    If the representing chain for
    $\Delta^q_{\alpha} \times_{B_{i_n}^{f_1}} \overline{\mathcal{M}}_{F_{21}}(B_{i_n}^{f_1},B^{f_2}) \in C_*^{F_{21}}$
    is of the form
    \[
        s_{\Delta^q_{\alpha} \times_{B_{i_n}^{f_1}} \overline{\mathcal{M}}_{F_{21}}(B_{i_n}^{f_1},B^{f_2})}%
        = \sum_{\beta} m_{\beta} \tau_{\beta} \in \mathcal{R}^{F_{21}}, 
    \]
    then the representing chain for the fibered product
    $Q \times_{B_{i_n}^{f_1}} \overline{\mathcal{M}}_{F_{21}}(B_{i_n}^{f_1},B^{f_2})$
    coincides with
    \[
        s_{Q \times_{B_{i_n}^{f_1}} \overline{\mathcal{M}}_{F_{21}}(B_{i_n}^{f_1},B^{f_2})}%
        = \sum_{\alpha,\,\beta} n_{\alpha} m_{\beta} \bigl((\sigma_{\alpha} \times \mathrm{id}) \circ \tau_{\beta} \bigr).
        \]

    \item Let $P\in C_p^{F_{21}}$ be a connected component of a fibered product of the form 
    \begin{align*}
        \dotDelta%
        \times_{B_{i_1}^{f_1}} \overline{\mathcal{M}}(B_{i_1}^{f_1},B_{i_2}^{f_1})%
        \times_{B_{i_2}^{f_1}} \cdots%
        \times_{B_{i_{n-1}}^{f_1}} \overline{\mathcal{M}}(B_{i_{n-1}}^{f_1},B_{i_n}^{f_1})%
        \times_{B_{i_n}^{f_1}} \overline{\mathcal{M}}_{F_{21}}(B_{i_n}^{f_1},B_{j_1}^{f_2}) \\
        \times_{B_{j_1}^{f_2}} \overline{\mathcal{M}}(B_{j_1}^{f_2},B_{j_2}^{f_2})%
        \times_{B_{j_2}^{f_2}} \cdots%
        \times_{B_{j_{k-1}}^{f_2}} \overline{\mathcal{M}}(B_{j_{k-1}}^{f_2},B_{j_k}^{f_2}),
    \end{align*}
    where $\dotDelta$ is a face of $\Delta^N$.
    Let $s^{F_{21}}_P = \sum_{\alpha} n_{\alpha} \sigma_{\alpha}\in \mathcal{R}^{F_{21}}$
    be the representing chain for $P$.
    For each $\alpha$ and any critical submanifold $B^{f_2}$ of $f_2$
    consider the fibered product $\Delta^p_{\alpha} \times_{B_{j_k}^{f_2}} \overline{\mathcal{M}}(B_{j_k}^{f_2},B^{f_2})$
    of the composition map
    \[
        \mathrm{ev}_+ \circ \mathrm{pr} \circ \sigma_{\alpha} \colon%
        \Delta^p_{\alpha} \xrightarrow{\sigma_{\alpha}} P \xrightarrow{\mathrm{pr}} \overline{\mathcal{M}}(B_{j_{k-1}}^{f_2},B_{j_k}^{f_2})%
        \xrightarrow{\mathrm{ev}_+} B_{j_k}^{f_2},
    \]
    and the beginning map $\mathrm{ev}_-\colon\overline{\mathcal{M}}(B_{j_k}^{f_2},B^{f_2})\to B_{j_k}^{f_2}$.
    If the representing chain for
    $\Delta^p_{\alpha} \times_{B_{j_k}^{f_2}} \overline{\mathcal{M}}(B_{j_k}^{f_2},B^{f_2})\in C_*^{f_2}$
    is of the form
    \[
        s_{\Delta^p_{\alpha} \times_{B_{j_k}^{f_2}} \overline{\mathcal{M}}(B_{j_k}^{f_2},B^{f_2})}%
        = \sum_{\beta} m_{\beta} \tau_{\beta} \in \mathcal{R}^{f_2}, 
    \]
    then the representing chain for the fibered product $P \times_{B_{j_k}^{f_2}} \overline{\mathcal{M}}(B_{j_k}^{f_2},B^{f_2})$
    coincides with
    \[
        s_{P \times_{B_{j_k}^{f_2}} \overline{\mathcal{M}}(B_{j_k}^{f_2},B^{f_2})}%
        = \sum_{\alpha,\,\beta} n_{\alpha} m_{\beta} \bigl((\sigma_{\alpha} \times \mathrm{id}) \circ \tau_{\beta} \bigr).
    \]
\end{enumerate}
\end{definition}

\begin{lemma}\label{lemma:exsitence of a good rep. chains for F}
There exists a good representing chain system for $F_{21}$.
\end{lemma}

\begin{proof}
As in Lemma \ref{lemma:existence of a good rep. chains},
we construct a good representing chain systems for $F_{21}$
by choosing exactly one representing chain for each $P\in C^{F_{21}}_p$.
Here we only give the necessary changes.

For each connected component of a fibered product
$\dotDelta^0 \times_{B_i^{f_1}} \mathcal{M}_{F_{21}}(B_i^{f_1},B^{f_2})$,
we choose exactly one representing chain and then extend it inductively
to obtain a representing chain for a connected component of a fibered product
$\dotDelta^q \times_{B_i^{f_1}} \mathcal{M}_{F_{21}}(B_i^{f_1},B^{f_2})$.

Let $Q\in C^{f_1}_q$ be a connected component of a fibered product
\[
    \dotDelta^0%
    \times_{B_{i_1}^{f_1}} \overline{\mathcal{M}}(B_{i_1}^{f_1},B_{i_2}^{f_1})%
    \times_{B_{i_2}^{f_1}} \cdots%
    \times_{B_{i_{n-1}}^{f_1}} \overline{\mathcal{M}}(B_{i_{n-1}}^{f_1},B_{i_n}^{f_1}),
\]
and let $s_Q = \sum_{\alpha} n_{\alpha} \sigma_{\alpha} \in \mathcal{R}^{f_1}$.
Assume that $P \in C^{F_{21}}_p$ is a connected component of a fibered product
$Q \times_{B_{i_n}^{f_1}} \mathcal{M}_{F_{21}}(B_{i_n}^{f_1},B^{f_2})$.
Consider the fibered product $\Delta^q_{\alpha} \times_{B_{i_n}^{f_1}} \mathcal{M}_{F_{21}}(B_{i_n}^{f_1},B^{f_2})$
of the maps $\mathrm{ev}_+ \circ \mathrm{pr} \circ \sigma_{\alpha}$ and $\mathrm{ev}_-$,
and let $\sum_{\beta} m_{\beta} \tau_{\beta} \in \mathcal{R}^{F_{21}}$ be its representing chain.
We then define the representing chain for $P$ by
\[
    s_P = \sum_{\alpha,\,\beta} n_{\alpha} m_{\beta} \bigl((\sigma_{\alpha} \times \mathrm{id}) \circ \tau_{\beta}\bigr) \in \mathcal{R}^{F_{21}}.
\]

Let $Q\in C^{F_{21}}_q$ be a connected component of a fibered product of the form
$\dotDelta^0 \times_{B_{i_n}^{f_1}} \overline{\mathcal{M}}_{F_{21}}(B_{i_n}^{f_1},B_{j_1}^{f_2})$
and let $s_Q = \sum_{\alpha} n_{\alpha} \sigma_{\alpha} \in \mathcal{R}^{F_{21}}$.
Assume that $P\in C^{F_{21}}_p$ is a connected component of a fibered product
$Q \times_{B_{j_1}^{f_2}} \mathcal{M}(B_{j_1}^{f_2},B_{j_2}^{f_2})$.
Consider the fibered product $\Delta^q_{\alpha} \times_{B_{j_1}^{f_2}} \overline{\mathcal{M}}(B_{j_1}^{f_2},B_{j_2}^{f_2}) \in C^{f_2}_*$
and let $\sum_{\beta} m_{\beta} \tau_{\beta} \in \mathcal{R}^{f_2}$ be its representing chain.
We then define the representing chain for $P$ by
\[
    s_P = \sum_{\alpha,\,\beta} n_{\alpha} m_{\beta} \bigl((\sigma_{\alpha} \times \mathrm{id}) \circ \tau_{\beta}\bigr) \in \mathcal{R}^{F_{21}}.
\]

For a general $P\in C^{F_{21}}_p$ we choose its representing chain
by extending the representing chains selected above in an inductive manner.
\end{proof}

We fix a good representing chain system $\mathcal{R}^{F_{21}}$ for $F_{21}$.
We define a homomorphism
\[
    (F_{21})_{\#}\colon (\CB_*(f_1,\mathcal{R}^{f_1}),\bm{\partial}_*) \to (\CB_*(f_2,\mathcal{R}^{f_2}),\bm{\partial}_*)
\]
as follows.
Let $\sigma\colon P\to B_i^{f_1}$ be a singular $C^{f_1}_p$-space in $B_i^{f_1}$.
For each $j=0$,~1, \dots,~$\dim{M}$
let $\sum_{\alpha} n^j_{\alpha} \sigma^j_{\alpha} \in \mathcal{R}^{F_{21}}$
be the representing chain for the fibered product $P\times_{B_i^{f_1}} \overline{\mathcal{M}}_{F_{21}}(B_i^{f_1},B_j^{f_2})$.
We then define
\begin{align*}
    F_{21}^j(\sigma) = \sum_{\alpha} n^j_{\alpha} (\mathrm{ev}_+\circ\mathrm{pr}\circ\sigma^j_{\alpha})%
    \quad \text{and}\quad
    (F_{21})_{\#}(\sigma) = \bigoplus_{j=0}^{\dim{M}} F_{21}^j(\sigma),
\end{align*}
which is well-defined due to the second condition of Definition \ref{definition:good family of rep. chain for F}.

A straightforward computation shows the following lemma (cf.\ \cite[Proposition 6.11]{BH10}).
However, the third condition of Definition \ref{definition:good family of rep. chain for F} is needed for the proof.

\begin{lemma}\label{lemma:chain map between MB multi-complexes}
The map $(F_{21})_{\#} \colon (\CB_*(f_1,\mathcal{R}^{f_1}),\bm{\partial}_*) \to (\CB_*(f_2,\mathcal{R}^{f_2}),\bm{\partial}_*)$
is a chain map and hence
it induces a homomorphism $(F_{21})_* \colon \HB_*(f_1,\mathcal{R}^{f_1}) \to \HB_*(f_2,\mathcal{R}^{f_2})$.
\end{lemma}


\subsection{Chain homotopy}\label{section:chain homotopy}

Let $f_1$, $f_2$, $f_3$, $f_4\colon M\to \RR$ be Morse--Bott--Smale functions on $(M,g)$.
For each $\heartsuit\in\{1,2,3,4\}$ and $i=0$,~1, \dots,~$\dim{M}$
let $B_i^{f_{\heartsuit}}$ denote the critical submanifold of $f_{\heartsuit}$ of Morse--Bott index $i$.
We assume that the components of $B_i^{f_{\heartsuit}}$ have the same dimension.
Let $b_i^{f_{\heartsuit}}=\dim{B_i^{f_{\heartsuit}}}$.

Let $H\colon M\times\RR\times\RR\to \RR$ be a smooth function which is strictly decreasing in the second and third component,
and satisfies
\[
    H(\cdot,s,t) = \begin{cases}%
        f_1 & \text{for $s,t\ll -1$}, \\
        f_2 - \text{(constant)} & \text{for $s\gg 1$ and $t\ll -1$}, \\
        f_3 - \text{(constant)} & \text{for $s\ll -1$ and $t\gg 1$}, \\
        f_4 - \text{(constant)} & \text{for $s,t\gg 1$}.
    \end{cases}
\]
We note that the function $H$ defines functions $F_{21},F_{42},F_{31},$ and $F_{43}$ as in Section \ref{section:chain map}.
Let $\{\varphi_{\alpha}\}_{\alpha\in\RR}$ denote the flow of $-\grad{H}$ with respect to the product metric on $M\times\RR\times\RR$.
Let $i$, $j=0$,~1, \dots,~$\dim{M}$.
One can define the moduli space $\mathcal{M}_H(B_i^{f_1},B_j^{f_4})$
which has a compactification $\overline{\mathcal{M}}_H(B_i^{f_1},B_j^{f_4})$.
This compactified moduli space $\overline{\mathcal{M}}_{H}(B _{i} ^{f_{1}},B _{j} ^{f_{4}})$
is either empty or a compact oriented smooth manifold with corners of dimension $b_i^{f_1}+i-j+1$.
Moreover, we define
\begin{align*}
    \partial\bigl(\overline{\mathcal{M}}_H(B_i^{f_1},B_j^{f_4})\bigr)%
    = (-1)^{i+b_i^{f_1}}%
    \Biggl(%
    & \sum_{n=0}^{\dim{M}}%
    \overline{\mathcal{M}}_{F_{21}}(B_i^{f_1},B_n^{f_2}) \times_{B_n^{f_2}} \overline{\mathcal{M}}_{F_{42}}(B_n^{f_2},B_j^{f_4}) \\
    &-\sum_{n=0} ^{\dim{M}}%
    \overline{\mathcal{M}}_{F_{31}}(B_i^{f_1},B_n^{f_3}) \times_{B_n^{f_3}} \overline{\mathcal{M}}_{F_{43}}(B_n^{f_3},B_j^{f_4}) \\
    &+\sum_{n<i}%
    \overline{\mathcal{M}}(B_i^{f_1},B_n^{f_1}) \times_{B_n^{f_1}} \overline{\mathcal{M}}_H(B_n^{f_1},B_j^{f_4}) \\
    &+\sum_{n>j}%
    \overline{\mathcal{M}}_H(B_i^{f_1},B_n^{f_4}) \times_{B_n^{f_4}} \overline{\mathcal{M}}(B_n^{f_4},B_j^{f_4})%
    \Biggr).
\end{align*}

Let $N>\dim{M}$.
For $p\geq 0$
let $C^H_p$ be the set consisting of $p$-faces of $\Delta^N$
and $p$-dimensional connected components of fibered products of the form
\begin{align*}
    \dotDelta^q%
    \times_{B_{i_1}^{f_1}} \overline{\mathcal{M}}(B_{i_1}^{f_1},B_{i_2}^{f_1})%
    \times_{B_{i_2}^{f_1}} \cdots%
    \times_{B_{i_{n-1}}^{f_1}} \overline{\mathcal{M}}(B_{i_{n-1}}^{f_1},B_{i_n}^{f_1})%
    \times_{B_{i_n}^{f_1}} \overline{\mathcal{M}}_H(B_{i_n}^{f_1},B_{j_1}^{f_4}) \\
    \times_{B_{j_1}^{f_4}} \overline{\mathcal{M}}(B_{j_1}^{f_4},B_{j_2}^{f_4})%
    \times_{B_{j_2}^{f_4}} \cdots%
    \times_{B_{j_{k-1}}^{f_4}} \overline{\mathcal{M}}(B_{j_{k-1}}^{f_4},B_{j_k}^{f_4}),
\end{align*}
or
\begin{align*}
    \dotDelta^q%
    \times_{B_{i_1}^{f_1}} \overline{\mathcal{M}}(B_{i_1}^{f_1},B_{i_2}^{f_1})%
    \times_{B_{i_2}^{f_1}} \cdots%
    \times_{B_{i_{n-1}}^{f_1}} \overline{\mathcal{M}}(B_{i_{n-1}}^{f_1},B_{i_n}^{f_1})%
    \times_{B_{i_n}^{f_1}} \overline{\mathcal{M}}_{F_{\heartsuit 1}}(B_{i_n}^{f_1},B_{j_1}^{f_{\heartsuit}}) \\
    \times_{B_{j_1}^{f_{\heartsuit}}} \overline{\mathcal{M}}(B_{j_1}^{f_{\heartsuit}},B_{j_2}^{f_{\heartsuit}})%
    \times_{B_{j_2}^{f_{\heartsuit}}} \cdots%
    \times_{B_{j_{k-1}}^{f_{\heartsuit}}} \overline{\mathcal{M}}(B_{j_{k-1}}^{f_{\heartsuit}},B_{j_k}^{f_{\heartsuit}})%
    \times_{B_{j_k}^{f_{\heartsuit}}} \overline{\mathcal{M}}_{F_{4\heartsuit}}(B_{j_k}^{f_{\heartsuit}},B_{r_1}^{f_4}) \\
    \times_{B_{r_1}^{f_4}} \overline{\mathcal{M}}(B_{r_1}^{f_4},B_{r_2}^{f_4})%
    \times_{B_{r_2}^{f_4}} \cdots%
    \times_{B_{r_{s-1}}^{f_4}} \overline{\mathcal{M}}(B_{r_{s-1}}^{f_4},B_{r_s}^{f_4}),
\end{align*}
where $\heartsuit\in\{2,3\}$.
Let $S^H_p=\ZZ[C^H_p]$.
For $p<0$ or $C^H_p=\emptyset$ we set $S^H_p=0$.
Then the pair $(S^H_*,\partial_*)$ is a chain complex of abstract topological chains.

A representing chain system for $H$ is also defined in the same way as in Definition \ref{definition:rep.sys.}.
Let $\heartsuit\in\{2,3\}$.
We fix good representing chain systems
$\mathcal{R}^{f_1}$, $\mathcal{R}^{f_2}$, $\mathcal{R}^{F_{\heartsuit 1}}$, and $\mathcal{R}^{F_{4\heartsuit}}$
for $f_1$, $f_2$, $F_{\heartsuit 1}$, and $F_{4\heartsuit}$, respectively.

\begin{definition}\label{definition:good family of rep. chain for H}
A representing chain system $\mathcal{R}^H$ for $H$ is called \textit{good}
if it satisfies the following conditions.

\begin{enumerate}
    \item For any $P\in C^H_p$
    there exists a unique element $(P,s_P)\in\mathcal{R}^H$
    such that $s_P$ represents the positive relative fundamental class of $P$ in $H_p(P,\partial P)$.

    \item Let $Q\in C_q^{f_1}$ be a connected component of a fibered product of the form 
    \[
       \dotDelta%
       \times_{B_{i_1}^{f_1}} \overline{\mathcal{M}}(B_{i_1}^{f_1},B_{i_2}^{f_1})%
       \times_{B_{i_2}^{f_1}} \cdots%
       \times_{B_{i_{n-1}}^{f_1}} \overline{\mathcal{M}}(B_{i_{n-1}}^{f_1},B_{i_n}^{f_1}),
    \]
    where $\dotDelta$ is a face of $\Delta^N$.
    Let $s^{f_1}_Q = \sum_{\alpha} n_{\alpha} \sigma_{\alpha} \in \mathcal{R}^{f_1}$
    be the representing chain for $Q$.
    For each $\alpha$ and any critical submanifold $B^{f_4}$ of $f_4$
    consider the fibered product $\Delta^q_{\alpha} \times_{B_{i_n}^{f_1}} \overline{\mathcal{M}}_H(B_{i_n}^{f_1},B^{f_4})$.
    If its representing chain is of the form $\sum_{\beta} m_{\beta} \tau_{\beta} \in \mathcal{R}^H$,
    then the representing chain for the fibered product
    $Q \times_{B_{i_n}^{f_1}} \overline{\mathcal{M}}_H(B_{i_n}^{f_1},B^{f_4})$
    coincides with
    \[
        s_{Q \times_{B_{i_n}^{f_1}} \overline{\mathcal{M}}_H(B_{i_n}^{f_1},B^{f_4})}%
        = \sum_{\alpha,\,\beta} n_{\alpha} m_{\beta} \bigl((\sigma_{\alpha} \times \mathrm{id}) \circ \tau_{\beta} \bigr).
    \]

    \item Let $P\in C_p^H$ be a connected component of a fibered product of the form 
    \begin{align*}
        \dotDelta^q%
        \times_{B_{i_1}^{f_1}} \overline{\mathcal{M}}(B_{i_1}^{f_1},B_{i_2}^{f_1})%
        \times_{B_{i_2}^{f_1}} \cdots%
        \times_{B_{i_{n-1}}^{f_1}} \overline{\mathcal{M}}(B_{i_{n-1}}^{f_1},B_{i_n}^{f_1})%
        \times_{B_{i_n}^{f_1}} \overline{\mathcal{M}}_H(B_{i_n}^{f_1},B_{j_1}^{f_4}) \\
        \times_{B_{j_1}^{f_4}} \overline{\mathcal{M}}(B_{j_1}^{f_4},B_{j_2}^{f_4})%
        \times_{B_{j_2}^{f_4}} \cdots%
        \times_{B_{j_{k-1}}^{f_4}} \overline{\mathcal{M}}(B_{j_{k-1}}^{f_4},B_{j_k}^{f_4}),
    \end{align*}
    where $\dotDelta$ is a face of $\Delta^N$.
    Let $s^H_P = \sum_{\alpha} n_{\alpha} \sigma_{\alpha}\in \mathcal{R}^H$
    be the representing chain for $P$.
    For each $\alpha$ and any critical submanifold $B^{f_4}$ of $f_4$
    consider the fibered product $\Delta^p_{\alpha} \times_{B_{j_k}^{f_4}} \overline{\mathcal{M}}(B_{j_k}^{f_4},B^{f_4})$.
    If its representing chain is of the form $\sum_{\beta} m_{\beta} \tau_{\beta} \in \mathcal{R}^{f_4}$,
    then the representing chain for the fibered product $P \times_{B_{j_k}^{f_4}} \overline{\mathcal{M}}(B_{j_k}^{f_4},B^{f_4})$
    coincides with
    \[
        s_{P \times_{B_{j_k}^{f_4}} \overline{\mathcal{M}}(B_{j_k}^{f_4},B^{f_4})}%
        = \sum_{\alpha,\,\beta} n_{\alpha} m_{\beta} \bigl((\sigma_{\alpha} \times \mathrm{id}) \circ \tau_{\beta} \bigr).
    \]

    \item Let $P\in C^{H} _{p}$ be a connected component of a fibered product of the form
    \begin{align*}
    \dotDelta^q%
    \times_{B_{i_1}^{f_1}} \overline{\mathcal{M}}(B_{i_1}^{f_1},B_{i_2}^{f_1})%
    \times_{B_{i_2}^{f_1}} \cdots%
    \times_{B_{i_{n-1}}^{f_1}} \overline{\mathcal{M}}(B_{i_{n-1}}^{f_1},B_{i_n}^{f_1})%
    \times_{B_{i_n}^{f_1}} \overline{\mathcal{M}}_{F_{\heartsuit 1}}(B_{i_n}^{f_1},B_{j_1}^{f_{\heartsuit}}) \\
    \times_{B_{j_1}^{f_{\heartsuit}}} \overline{\mathcal{M}}(B_{j_1}^{f_{\heartsuit}},B_{j_2}^{f_{\heartsuit}})%
    \times_{B_{j_2}^{f_{\heartsuit}}} \cdots%
    \times_{B_{j_{k-1}}^{f_{\heartsuit}}} \overline{\mathcal{M}}(B_{j_{k-1}}^{f_{\heartsuit}},B_{j_k}^{f_{\heartsuit}})%
    \times_{B_{j_k}^{f_{\heartsuit}}} \overline{\mathcal{M}}_{F_{4\heartsuit}}(B_{j_k}^{f_{\heartsuit}},B_{r_1}^{f_4}) \\
    \times_{B_{r_1}^{f_4}} \overline{\mathcal{M}}(B_{r_1}^{f_4},B_{r_2}^{f_4})%
    \times_{B_{r_2}^{f_4}} \cdots%
    \times_{B_{r_{s-1}}^{f_4}} \overline{\mathcal{M}}(B_{r_{s-1}}^{f_4},B_{r_s}^{f_4}).
    \end{align*}
    Then its representing chain $s_P$ is compatible with the good representing chain systems
    $\mathcal{R}^{f_1}$, $\mathcal{R}^{F_{\heartsuit 1}}$, $\mathcal{R}^{f_\heartsuit}$, and $\mathcal{R}^{F_{4\heartsuit}}$.
\end{enumerate}
\end{definition}

As with Lemmas \ref{lemma:existence of a good rep. chains} and \ref{lemma:exsitence of a good rep. chains for F},
one has the following.

\begin{lemma}\label{lemma:exsitence of a good rep. chains for H}
There exists a good representing chain system for $H$.
\end{lemma}

We fix a good representing chain system $\mathcal{R}^H$ for $H$.
We define a homomorphism
\[
    H_{\#}\colon (\CB_*(f_1,\mathcal{R}^{f_1}),\bm{\partial}_*) \to (\CB_*(f_4,\mathcal{R}^{f_4}),\bm{\partial}_*)
\]
as follows.
Let $\sigma\colon P\to B_i^{f_1}$ be a singular $C^{f_1}_p$-space in $B_i^{f_1}$.
For each $j=0$,~1, \dots,~$\dim{M}$
let $\sum_{\alpha} n^j_{\alpha} \sigma^j_{\alpha} \in \mathcal{R}^H$
be the representing chain for the fibered product $P\times_{B_i^{f_1}} \overline{\mathcal{M}}_H(B_i^{f_1},B_j^{f_4})$.
We then define
\begin{align*}
    H^j(\sigma) = \sum_{\alpha} n^j_{\alpha} (\mathrm{ev}_+\circ\mathrm{pr}\circ\sigma^j_{\alpha})%
    \quad \text{and}\quad
    H_{\#}(\sigma) = \bigoplus_{j=0}^{\dim{M}} H^j(\sigma).
\end{align*}

A similar argument to the proof \cite[Theorem 6.13]{BH10} and
the third condition of Definition \ref{definition:good family of rep. chain for H} shows the following lemma.

\begin{lemma}\label{lemma:chain homotopy}
The map $H_{\#} \colon (\CB_*(f_1,\mathcal{R}^{f_1}),\bm{\partial}_*) \to (\CB_*(f_4,\mathcal{R}^{f_4}),\bm{\partial}_*)$ satisfies
\[
    (F_{43})_{\#} \circ (F_{31})_{\#} - (F_{42})_{\#} \circ (F_{21})_{\#} = \partial H_{\#} + H_{\#} \partial.
\]
\end{lemma}

Now the exactly same argument in \cite[Corollaries 6.14--6.16]{BH10} completes the proof of Theorem \ref{theorem:MB homology theorem}.







\bibliographystyle{amsalpha}
\bibliography{orita_bibtex}
\end{document}